\documentclass[a4paper]{article}

\usepackage[english,russian]{babel}
\usepackage[utf8x]{inputenc}
\usepackage[T1]{fontenc}

\usepackage[a4paper,top=3cm,bottom=2cm,left=3cm,right=3cm,marginparwidth=1.75cm]{geometry}

\usepackage{amsfonts,amsmath,latexsym,amssymb,amsthm,tipa}
\usepackage{graphicx}
\usepackage[colorinlistoftodos]{todonotes}
\usepackage[colorlinks=true, allcolors=blue]{hyperref}
\usepackage{booktabs}

\bibliography{bibliography}

\title{Производящие семейства на Арбореллевских графах}
\author{Иван Яковлев\footnote{Работа выполнена при поддержке  Лаборатории зеркальной симметрии НИУ ВШЭ, грант Правительства РФ Договор  14.641.31.0001}}

\begin{document}
\newtheorem{prop}{Proposition}
\newtheorem{mydef}{Definition}
\newtheorem{conj}{Conjecture}
\newtheorem{comm}{Comment}
\newtheorem{data}{Symbol}
\newtheorem{ex}{Example} 
\newtheorem{question}{Question}
\newtheorem{prof}{Proof}
\newtheorem{cons}{Consequence}
\newtheorem{const}{Construction}
\newtheorem{res}{Result}
\newtheorem{lem}{Lemma}
\newtheorem{theorem}{Theorem}

\maketitle
\begin{abstract}
    The paper is devoted to the study of exact curves on Arborealized Liouville surfaces. We introduce the notion of a generating family for such curves. Our main statement is a hamiltonian lifting property: the set of curves admitting a generating family is closed with respect to Hamiltonian isotopes. \\
    This is part of the author's future thesis. It will be translated into English within a few weeks. In the future, we plan to generalize our results in several directions and find applications.
    \newline
    Работа посвящена изучению точных кривых на Арбореализованных Лиувиллевых поверхностях. Мы вводим понятие производящего семейства для таких кривых. Наше основное утверждение это свойство гамильтонова подъема: множество кривых, допускающих производящее семейство, замкнуто относительно гамильтоновых изотопий. \\
    Это часть будущей диссертации автора. В течении нескольких недель она будет переведена на английский язык. В дальнейшем, мы планируем обобщить наши результаты в нескольких направлениях и найти приложения.\\
    Ключевые слова: Симплектические поверхности, Производящие семейства, Арбореллевские пространства.
\end{abstract}
\section{Введение}
Кокасательные расслоения выделены среди других объектов симплектической топологии.\\
В частности, для них развита техника производящих семейств, позволяющая строить Лагранжевы подмногообразия и изучать их в терминах гладкой геометрии базы. В этой статье мы обобщаем ее на некомпактные симплектические поверхности, а в будущем планируем разобрать случай  Арбореализуемых Лиувиллевых многообразий произвольной размерности. 

Поверхности, с которыми мы работаем, можно рассматривать как обобщенные кокасательные расслоения от некоторых графов. \textbf{Арбореллевским графом} мы называем ленточный граф, валентность вершин которого не превышает трех, на котором для каждой вершины $v$ валентности $3$ зафиксирована ее \textbf{"ножка"{}}, то есть одно из инцидентных ей ребер.\\
Стартуя с Арбореллевского графа мы строим поверхность, в которую граф вкладывается так, что поверхность ретракируется на образ вложения. Она склеивается из стандартных плиток. \\
\begin{center}
\includegraphics[width=6cm]{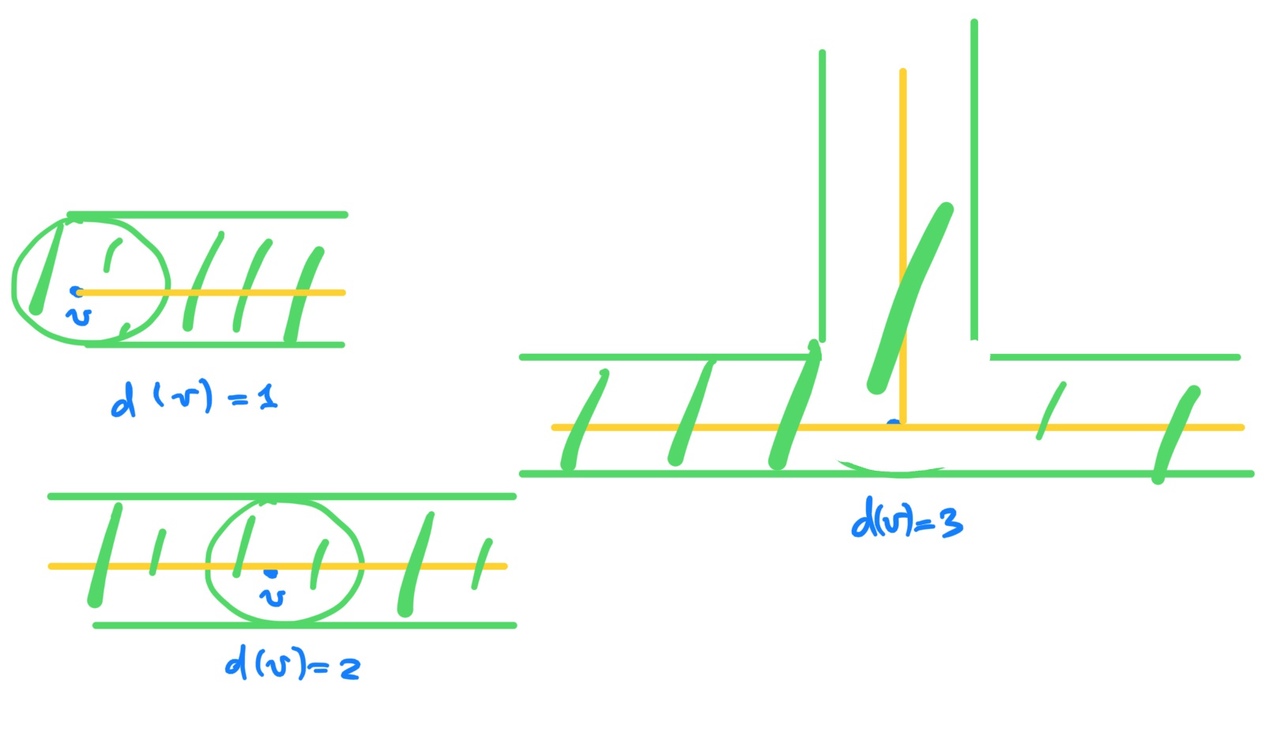}
\end{center}
На поверхности $S(\mathcal{T})$ задана каноническая симплектическая форма. Мы рассматриваем это симплектическое многообразие как аналог кокасательного расслоения от графа $\mathcal{T}$. Эта частный случай конструкции, которая изучалась в статье~\cite{Arborealization3} для Арбореллевских пространств произвольной размерности. Мы подробно описываем ее для графов в обзорном разделе~\ref{sec:surface}. 

Нас интересуют точные кривые на поверхностях $S(\mathcal{T})$. Если мы работаем на плоскости или на цилиндре, хорошим источником кривых являются производящие семейства. В общем случае, поверхность $S(\mathcal{T})$ допускает покрытие картами $U_v$, пронумерованными вершинами графа $\mathcal{T}$. На каждой из них задан выделенный симплектоморфизм с кокасательным пространством от интервала $I_v$, полученного объединением ребер, инцидентных вершине $v$ и отличных от ее шляпки. Используя это покрытие, мы сопоставим каждой кривой на $\mathbb{S}(\mathcal{T})$ набор кривых в кокасательных пространствах к ребрам $\mathcal{T}$. Скажем, что кривая \textbf{локально допускает производящие семейства}, если это верно для ее пересечений с любой картой. Мы обсуждаем эти определения в разделе~\ref{sec:curves}.

Чтобы глобализовать это определение, нам необходимо наложить условия согласованности на производящие семейства для соседних карт $U_v$ и $U_w$. Если общее ребро вершин $v$ и $w$ не является ножкой ни для одной из них, это сделать просто: достаточно потребовать, чтобы ограничения этих семейств на интервал ребра были стабильно эквивалентны. Если же $e$ является ножкой, для определения ограничения приходится вводить оператор поворота $\circlearrowleft$. В разделе~\ref{sec:chekanov} мы приведем общую конструкцию, позволяющую получить этот оператор и проверить его свойства. Это главное техническое место статьи. \\
В последнем разделе~\ref{sec:main} мы наконец определим производящие семейства на Арбореллевских графах. Главным инструментом в работе с производящими семействами является свойство гамильтонова подъема. Мы доказываем его аналог для Арбореллевских графов.\\
Основной результат этой статьи это следующее утверждение:
\begin{theorem}
\label{theorem:main}
Рассмотрим кривую $\gamma\subset \Sigma(\mathcal{T})$, допускающую производящее семейство и 
\begin{equation*}
    \psi\in Ham(\Sigma(\mathcal{T})).
\end{equation*}
Ее образ при Гамильтоновой изотопии $\psi(\gamma)$ также допускает производящее семейство. 
\end{theorem}
Как мы указывали выше, Арбореллевские графы являются частным случаем Арбореллевских пространств, для которых точно также определены обобщенные кокасательные расслоения. В дальнейшем мы собираемся обобщить наши результаты на эту более общую ситуацию.

Теорема о поднятии гамильтоновой изотопии имеет множество применений. В том числе, из ее версии для семейств, квадратичных на бесконечности, выводится доказательство локальной версии Лагранжевой гипотезы Арнольда. Мы надеемся в будущем ввести аналог условия квадратичности на бесконечности и применить нашу теорему в этом контексте, чтобы получить обобщения локальной гипотезы Арнольда на Арбореллевские пространства.

\subsubsection*{Благодарности}
Я бы хотел поблагодарить П.Е. Пушкаря, который рассказал мне о производящих семействах, моих научных руководителей Л. Кацаркова из ВШЭ и А.А. Рослого из Сколтеха, а так же М.Э. Казаряна и Н.М. Курносова. 
\newpage
\section{Симплектическая геометрия открытых поверхностей}
\label{sec:surface}
\textbf{Поверхность} это двумерное вещественное многообразие без края конечного рода.
Рассмотрим некомпактную поверхность $S$ с симплектической структурой, заданной формой
\begin{equation*}
    \omega\in\Omega^2(S).
\end{equation*} 
Так как $H^2(S,\mathbb{R})=0$, форма объема $\omega$ точна. \textbf{Лиувиллева форма} это ее потенциал
\begin{equation*}
    \lambda\in \Omega^1(S,\mathbb{R}), \mbox{ } \omega=d\lambda.
\end{equation*}
\textbf{Лиувиллево векторное поле} двойственно $\lambda$ относительно невырожденного спаривания $\omega$
\begin{equation*}
    X\in\mathcal{X}(S), \mbox{ } \iota_{X} \omega= \lambda. 
\end{equation*} 
\begin{mydef}
\textbf{Точная симплектическая поверхность} $\mathbb{S}$ это симплектическое многообразие размерности $2$ с фиксированными Лиувиллевыми формой и векторным полем
\begin{equation*}
    \mathbb{S}=({S},\omega,\lambda, X), \mbox{ } \omega=d\lambda,\mbox{ } \iota_{X} \omega= \lambda.
\end{equation*}
Поверхность $\mathbb{S}$, на которой поле $X$ полно, называется \textbf{Лиувиллевой поверхностью}, если существует такое компактное подмногообразие с краем $\Sigma\subset S$, что поле $X$
\begin{itemize}
    \item не имеет нулей вне подобласти $\Sigma$,
    \item трансверсально $\partial \Sigma$ и "торчит наружу"{}.
\end{itemize}
\centering 
\includegraphics[width=7cm]{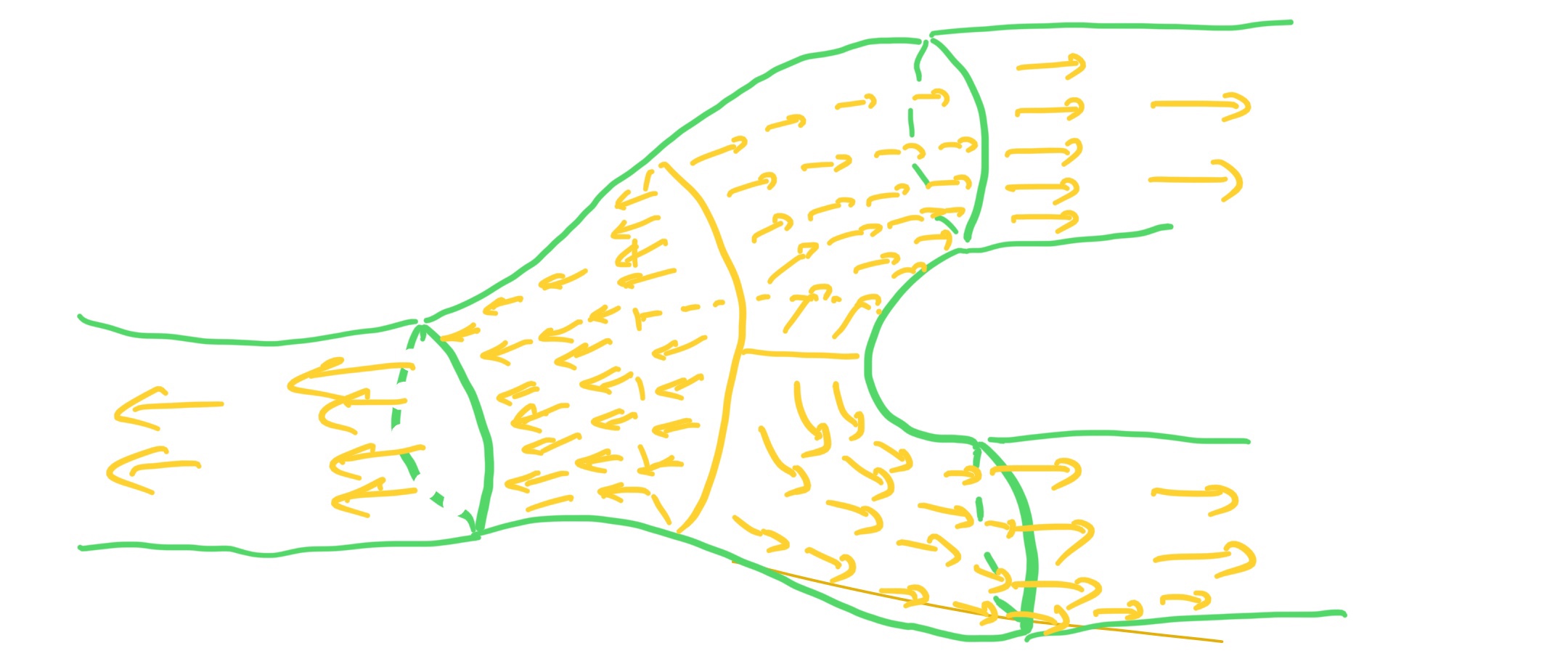}
\includegraphics[width=5cm]{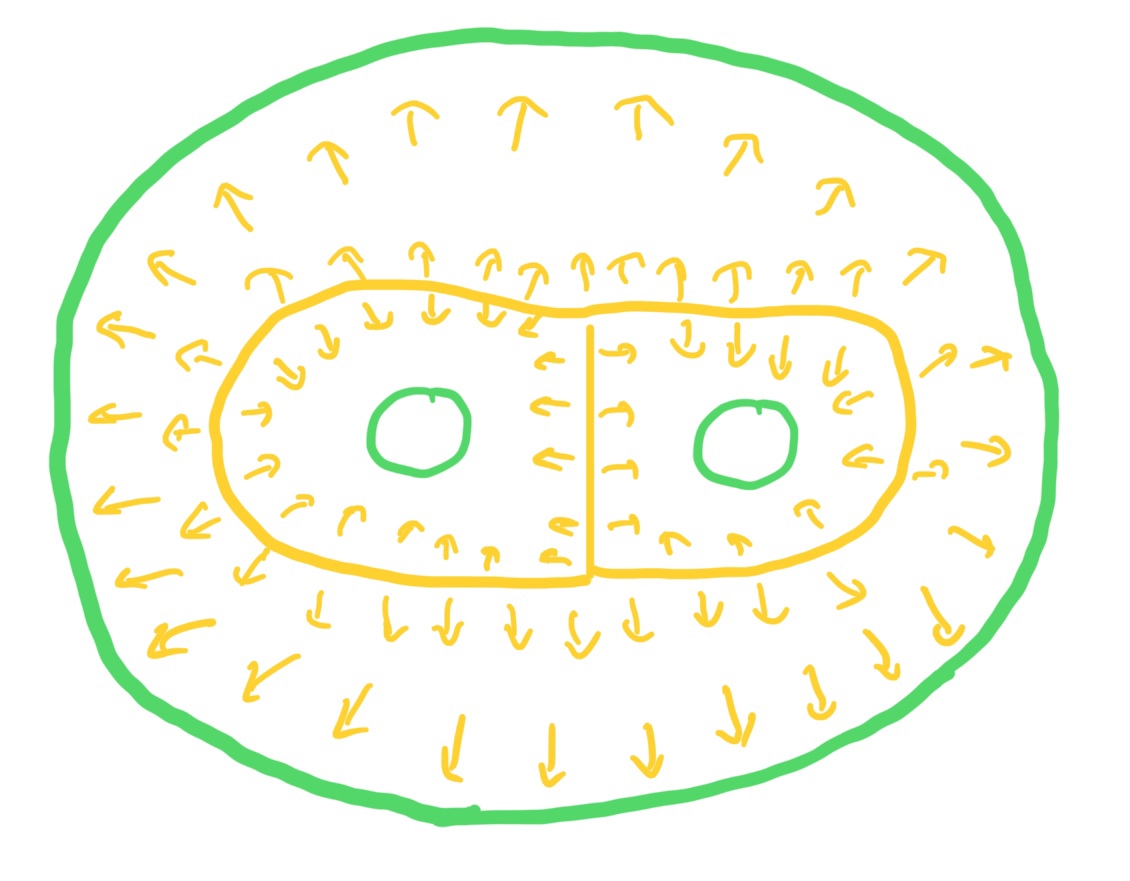}
\newline
На картинке изображена поверхность "штаны"{} со стандартной симплектической структурой. Желтыми стрелочками обозначены линии Лиувиллева векторного поля, а сплошной желтой линией - множество его нулей. 
\end{mydef}
Конечно, вся тройка тензоров $(\omega,\lambda, X)$ восстанавливаются по Лиувиллевой форме $\lambda$.\\
Изучение точных симплектических структур и Лиувиллевых многообразий (любой размерности) было начато в~\cite{EG}. 
Современный обзор теории см.  например~\cite{CE},~\cite{E},~\cite{Arborealization3}, раздел $2$.

\subsection*{Лиувиллевы структуры на поверхностях}
Определение Лиувиллевой поверхности аксиоматизирует свойства цилиндра $\mathbb{T}^*(S^1)$. 
\begin{ex} На кокасательном пространстве  $T^*S^1$ определена каноническая 1-форма $\theta$: 
\begin{equation*}
    \forall x\in S^1\mbox{, } v\in T^*_xS^1 \mbox{, } l\in T_{(x,v)}(T^*S^1), \mbox{ } \theta_{(x,v)}(l):=\pi^*v_{(x,v)}(l)=\langle D\pi_{(x,v)}(l), v\rangle.
\end{equation*}
В координатах форма $\theta$ имеет вид $pdq$, так что форма $\omega=d\theta=dp\wedge dq$ невырождена. Лиувиллево векторное поле записывается как $p\frac{\partial}{\partial p}$, так что оно полно. В качестве $\Sigma$ подойдет
\begin{equation*}
    \Sigma:=D^*S^1:=\{(x,v)\in T^*S^1|\mbox{ } |v|\le 1\}. 
\end{equation*}
Таким образом $\mathbb{T}^*(S^1)=(T^*S^1, dp\wedge dq, pdq, p\frac{\partial}{\partial p})$ является Лиувиллевой поверхностью.\\
\begin{center}
    \includegraphics[width=4cm]{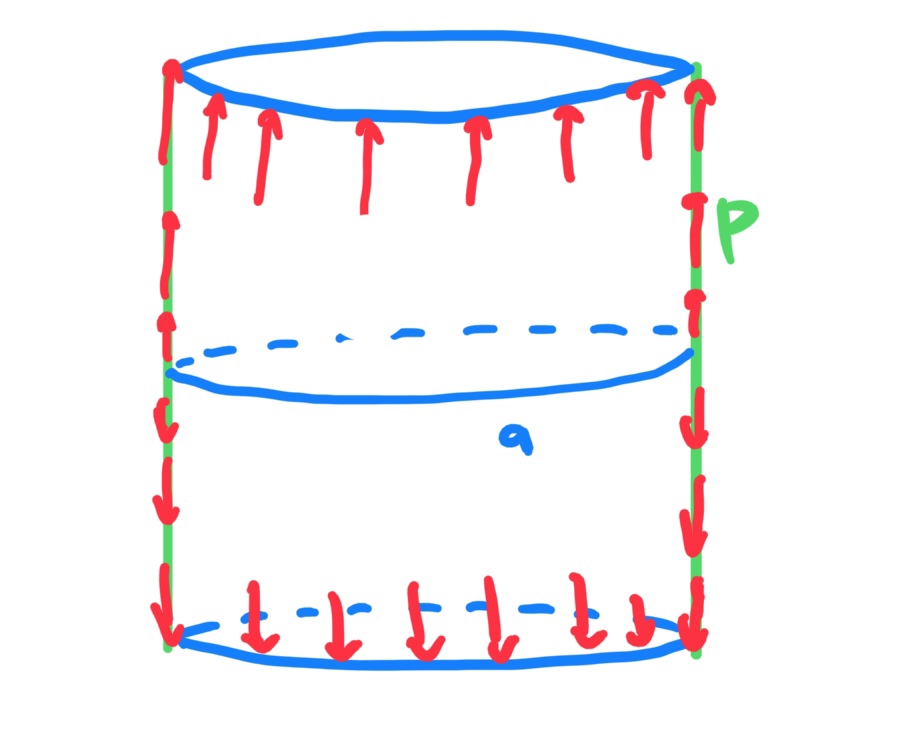}
\end{center}
На рисунке изображен цилиндр с координатами $q,p$. Красным отмечено  векторное поле $p\frac{\partial}{\partial p}$
\end{ex}
По определению, Лиувиллева поверхность ориентирована. Кроме того, она некомпактна: в противном случае, она имеет конечную площадь, которая оказывается равна нулю 
\begin{equation*}
    0<Vol \mbox (S)=\int\limits_{[S]}\omega=\langle [S],[\omega] \rangle=0.
\end{equation*}
Сейчас мы построим Лиувиллеву структуру на любой некомпактной ориентируемой поверхности $S$. Для начала введем на $S$ структуру римановой поверхности. Известно, что полученную кривую можно голоморфно и собственно вложить в аффинное пространство $\mathbb{C}^n$.
\begin{ex}
Рассмотрим замкнутую комплексную кривую $S\subset \mathbb{C}^n$. Ограничение 
\begin{equation*}
    f:=|z|^2|_{S}\in\mathcal{C}^\infty(S)
\end{equation*}
функции $|z|$ на $S$ определяет плюрисубгармоническую функцию $f$. Другими словами,
\begin{equation*}
\forall p\in S,\mbox{ } \forall v\in T_pS, \mbox{ }  -dd^{\mathbb{C}}f_z(v,Jv)>0,
\end{equation*}
то есть форма $\omega:=-dd^{\mathbb{C}}f$ невырождена. В качестве Лиувиллевой формы рассмотрим
\begin{equation*}
    \lambda:=-d^{\mathbb{C}}f.
\end{equation*}
Лиувиллево поле для $\lambda$ совпадает с градиентом $\nabla_{\rho}f$ относительно метрики $\rho=\omega(J\cdot,\cdot)$. В частности, оно полно. Рассмотрим $t\in \mathbb{R}$, превышающее все критические значения $f$
\begin{equation*}
     t> \max vCrit(f).
\end{equation*}
Тогда поверхность с краем $\Sigma:=f^{-1}((-\infty, t])$ удовлетворяет нашему условию, а значит 
\begin{equation*}
    \mathbb{S}:=(S,-dd^{\mathbb{C}}f,-d^{\mathbb{C}}f,\nabla_{\rho}f)
\end{equation*}
является Лиувиллевой поверхностью. Она называется \textbf{Штейновой поверхностью}.
\end{ex}
Таким образом, на $S$ существует Лиувиллева структура.\\
Достаточно хорошая структура Лиувиллевой поверхности на $S$ единственна. Мы не будем это доказывать, но определение изоморфизма Лиувиллевых поверхностей дадим. 
\begin{mydef}
\textbf{Изоморфизм} Лиувиллевых поверхностей $\Phi$ это такой диффеоморфизм 
\begin{equation*}
    \Phi:\mathbb{S}_1\to \mathbb{S}_2, \mbox{ что}
\end{equation*}
\begin{equation*}
   \Phi^*\lambda_2=\lambda_1+dh, \mbox{ } h\in\mathcal{C}^\infty(S_1)
\end{equation*}
где $h$ имеет компактный носитель. В частности, $\Phi$ это точный симплектоморфизм.
\end{mydef}
Напомним, что \textbf{гамильтоновым симплектоморфизмом} $\mathbb{S}$ называется диффеоморфизм 
\begin{equation*}
    \psi\in Ham(\mathbb{S}),
\end{equation*}
для которого найдется семейство $H_t\in\mathcal{C}^\infty(S), t\in [0,1]$, такое, что $\psi=\psi_1$, где 
\begin{equation*}
    \psi_t\in Diff(S), 
\end{equation*}
\begin{equation*}
    \frac{\partial}{\partial t}\psi_t=X_t\circ \psi_t,\mbox{ } \iota_{X_t}\omega = dH_t. 
\end{equation*}
\textbf{Носителем} диффеоморфизма $\psi$ 
называется замыкание $ supp(\psi):= \overline{\{x\in S|\mbox{ } \psi(x)\ne x\}}.$\\
Группа гамильтоновых симплектоморфизмов $\mathbb{S}$ с компактным носителем обозначается $Ham^c(\mathbb{S})$.
\begin{ex}
Рассмотрим $\psi\in Ham^c(\mathbb{S})$ и соответствующее семейство $H_t$. Тогда 
\begin{equation*}
    \mathcal{L}_{X_t}(\lambda)=\iota_{X_t}\omega+d\iota_{X_t}\lambda=d(H_t+\iota_{X_t}\lambda)\Rightarrow \psi_{t_0}^*(\lambda)=\lambda+d\int_0^{t_0}(H_t+\iota_{X_t}\lambda)dt.
\end{equation*}
Таким образом $\psi$ это точный симплектоморфизм. Так как $supp(\psi)$ компактен, а $\psi^*(\lambda)-\lambda$ тождественно равно $0$ вне $supp(\psi)$, $\psi$ определяет Лиувиллев автоморфизм поверхности $\mathbb{S}$.
\end{ex}
\subsection*{Скелет Лиувиллевой поверхности}
Рассмотрим Лиувиллеву поверхность $\mathbb{S}=({S},\omega,\lambda, X)$ и  соответствующую поверхность $\Sigma$. \\
По формуле Картана, производная вдоль $X$ сохраняет $\lambda$, а его поток $\phi_t$ растягивает $\lambda$
\begin{equation*}
    \mathcal{L}_X(\lambda)=\iota_Xd\lambda+d\iota_X\lambda=\iota_X\omega+ \iota_X\iota_X\omega=\lambda, 
\end{equation*}
\begin{equation*}
   \frac{\partial}{\partial t}\phi_t^*\lambda= \phi_{t_0}^*(\mathcal{L}_X(\lambda))=\phi_{t_0}^*\lambda\Rightarrow \phi_{t_0}^*\lambda = e^t\lambda. 
\end{equation*}
По теореме о трубчатой окрестности, он задает диффеоморфизм 
\begin{equation*}
      \Phi=(\phi_t,t): (\partial \Sigma \times\mathbb{R}_-,d(e^t\alpha),e^t\alpha,\frac{\partial}{\partial t})\to (Op(\partial \Sigma),\omega,\lambda,X).
  \end{equation*}
Используя эту тривиализацию, продолжим $(\omega,\lambda,X)$ на некомпактное многообразие 
\begin{equation*}
    \widehat{\Sigma}:=\Sigma\cup_{\partial \Sigma} \partial \Sigma\times [0,+\infty), \mbox{ } (\omega, \lambda, X)|_{\partial \Sigma\times\mathbb{R}}=(d(e^t\alpha),e^t\alpha,\frac{\partial}{\partial t}).
\end{equation*}
\begin{center}
\includegraphics[width=7cm]{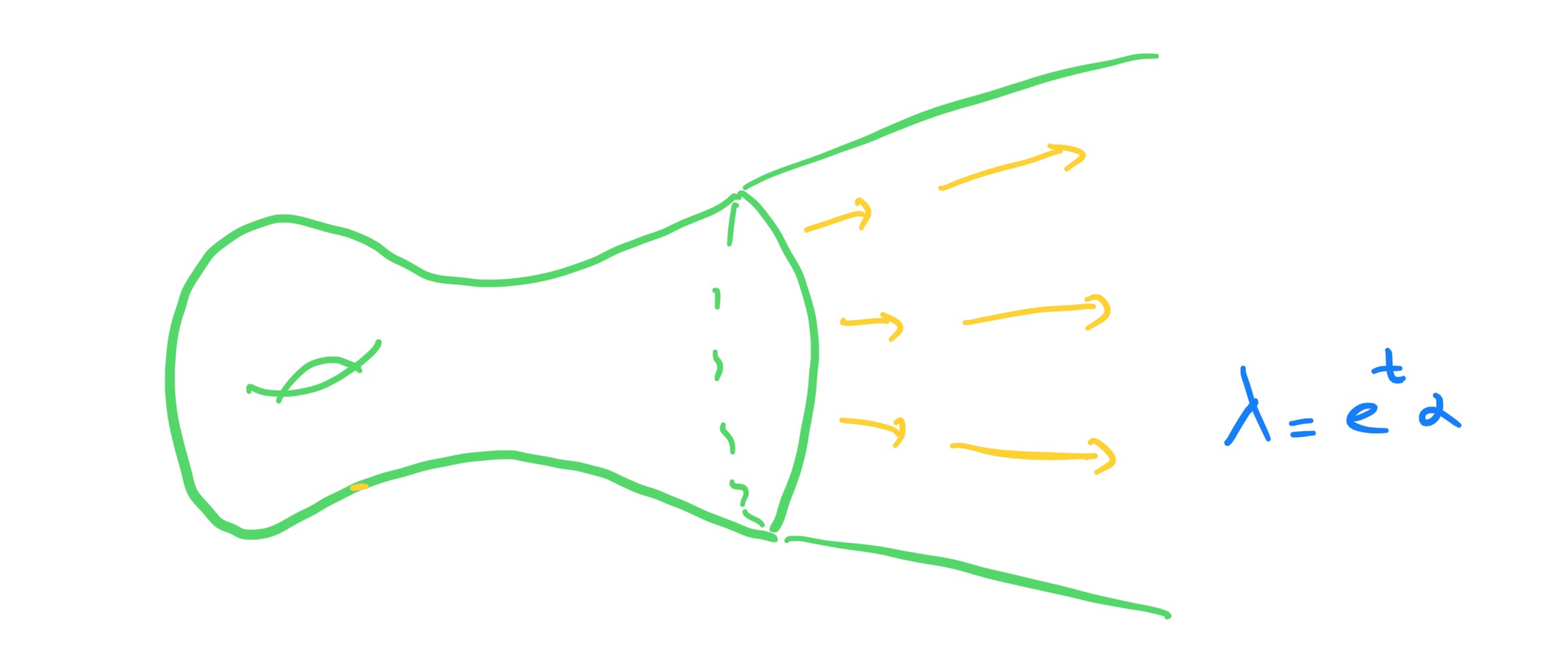}
\newline
\textit{Результат пополнения области, полученной выкидыванием открытого диска из тора.}
\end{center}
Поверхность с краем $\Sigma$ называется Лиувиллевой областью, а $\widehat{\Sigma}$ ее пополнением. Так как $X$ полное и не имеет нулей вне $\Sigma$, его поток определяет диффеоморфизм  $\widehat{\Sigma}$ на $S$. Поле $X$ действует на $\lambda$ и $\omega$ растяжениями, так что полученный диффеоморфизм сохраняет $(\omega,\lambda,X)$
\begin{equation*}
   \Phi: \mathbb{S}\cong \widehat{\Sigma}. 
\end{equation*}
Наоборот, дополнение до края Лиувиллевой области $\Sigma$ наделяется структурой Лиувиллевой поверхности, изоморфной $\widehat{\Sigma}$. Это позволяет думать про $\Sigma$ как про компактификацию $\mathbb{S}$. \\
Таким образом, $\mathbb{S}$ определяется произвольной Лиувиллевой подобластью $\Sigma\subset S$. Пересечение всех подобластей называется \textbf{скелетом} Лиувиллевой поверхности $\mathbb{S}$. Это аттрактор поля $-X$
\begin{equation*}
    Sk(\mathbb{S}):=\{p\in S|\mbox{ траектория } \phi_t(p)\mbox{ содержится в компакте} \} = \bigcap_{t>0}\phi_{-t}(\Sigma).
\end{equation*}
\begin{center}
\includegraphics[width=10cm]{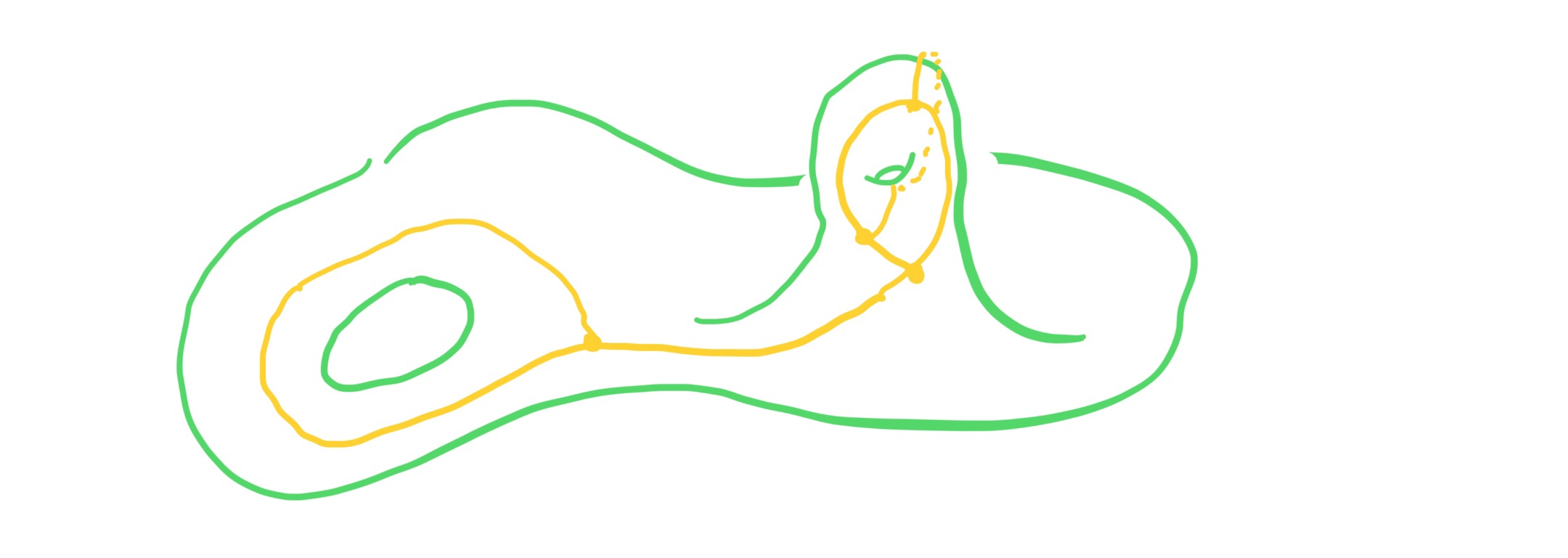}
\newline
\textit{Желтым изображен скелет зеленой области}
\end{center}
Скелет  Лиувиллевой поверхности $\mathbb{S}$ это компактное подмножество $S$ меры $0$. Поток $\phi_t$ Лиувиллева  поля $X$ сохраняет скелет $Sk(\mathbb{S})$ и ретракирует на него всю поверхность $S$. 
\begin{ex}
Скелет Лиувиллевой поверхности $\mathbb{T}^*S^1$ совпадает с нулевым сечением $0_{S^1}$.
\end{ex}
Мы хотим рассматривать $Sk(\mathbb{S})$ как аналог нулевого сечения Лиувиллевой поверхности, а саму поверхность $\mathbb{S}$ как "обобщенное кокасательное расслоение"{} от ее скелета. Наша задача - определить производящие семейства на $Sk(\mathbb{S})$ исходя из этой аналогии.\\ 
Теория производящих семейств основана на том, что гладкая геометрия $0_{S^1}$ определяет симплектическую геометрию $\mathbb{T}^*S^1$. Однако то, что скелет $\mathbb{T}^*Q$ гладкий, это аномалия.
\begin{ex}
Пусть $Sk(\mathbb{S})$ является гладким многообразием. Так как $Sk(\mathbb{S})$ вложен в $S$ и имеет меру $0$, его размерность не превышает $1$. Так как скелет связен и компактен, это отрезок, окружность или точка. Таким образом, поверхность $\mathbb{S}$ гомотопически эквивалентна $S^1$ или стягиваема, а значит она диффеоморфна цилиндру или плоскости.
\end{ex}
Для других $\mathbb{S}$ скелет не гладкий и, более того, может выглядеть очень плохо. \\
Однако, $Sk$ не инвариантен при Лиувиллевых изоморфизмах, так что можно поискать наиболее простую форму скелета для Лиувиллевых поверхностей, изоморфных $\mathbb{S}$.
\begin{ex}
Рассмотрим Штейнову поверхность $\mathbb{S}:=(S,-dd^{\mathbb{C}}f,-d^{\mathbb{C}}f,\nabla_{\rho}f)$. Так как $f$ растет вдоль потока $\nabla_{\rho}f$, траектория $\phi_t(x)$ содержится в компакте тогда и только тогда, когда $\lim \phi_t(x)$ является критической точкой $f$. Таким образом, скелет это объединение 
\begin{equation*}
    Sk(\mathbb{S}):=\bigcup_{p\in Crit(f)}W^{st}(p), \mbox{ } W^{st}(p):=\{x\in S|\mbox{ } \lim_{t\to\infty}\phi^t(x)=p\}
\end{equation*}
стабильных подмногообразий $W^{st}(p)$ для всех критических точек $p$ функции $f$.
\end{ex}
Так как плюрисубгармоничность это открытое условие, мы можем выбрать в качестве $f$ функцию Морса-Смейла. Тогда каждое из множеств $W^{st}(p)$ является объединением конечного числа градиентных траекторий, заканчивающихся в $p$. Следовательно, $ Sk(\mathbb{S})$ является \textbf{вложенным графом}, то есть таким графом $\Gamma$, каждое ребро которого гладко вложено в $S$.
\subsection*{Арбореллевские графы}
Ниже, мы определим производящее семейство на $\Gamma$ и сопоставим ему кривую в $\mathbb{S}$, $Sk(\mathbb{S})=\Gamma$. Чтобы это сделать, надо уметь восстанавливать поверхность $\mathbb{S}$ по ее скелету $\Gamma$. Лиувиллева поверхность $\mathbb{S}$ восстанавливается по произвольной открытой окрестности $Op(\Gamma)$, так как
\begin{equation*}
    \Sigma\subset Op(\Gamma)
\end{equation*}
для некоторой Лиувиллевой подобласти $\Sigma\subset S$ и $\mathbb{S}=\widehat{\Sigma}$. Таким образом, достаточно уметь восстанавливать $Op(\Gamma)$.  Если $\Gamma=S^1$, то по теореме Вайнштейна $Op(\Gamma)=\mathbb{T}^*(S^1)$, однако для произвольного графа восстановить ее окрестность невозможно. 
\begin{ex}
Рассмотрим граф $\Gamma\subset S$ с одной вершиной и двумя петлями. Его окрестность $Op(\Gamma)$ диффеоморфна либо проколотому тору, либо штанам. Это зависит от того, как подклеиваются ленточки, отвечающие ребрам, к дискам, отвечающим вершинам.

\centering
\includegraphics[width=10cm]{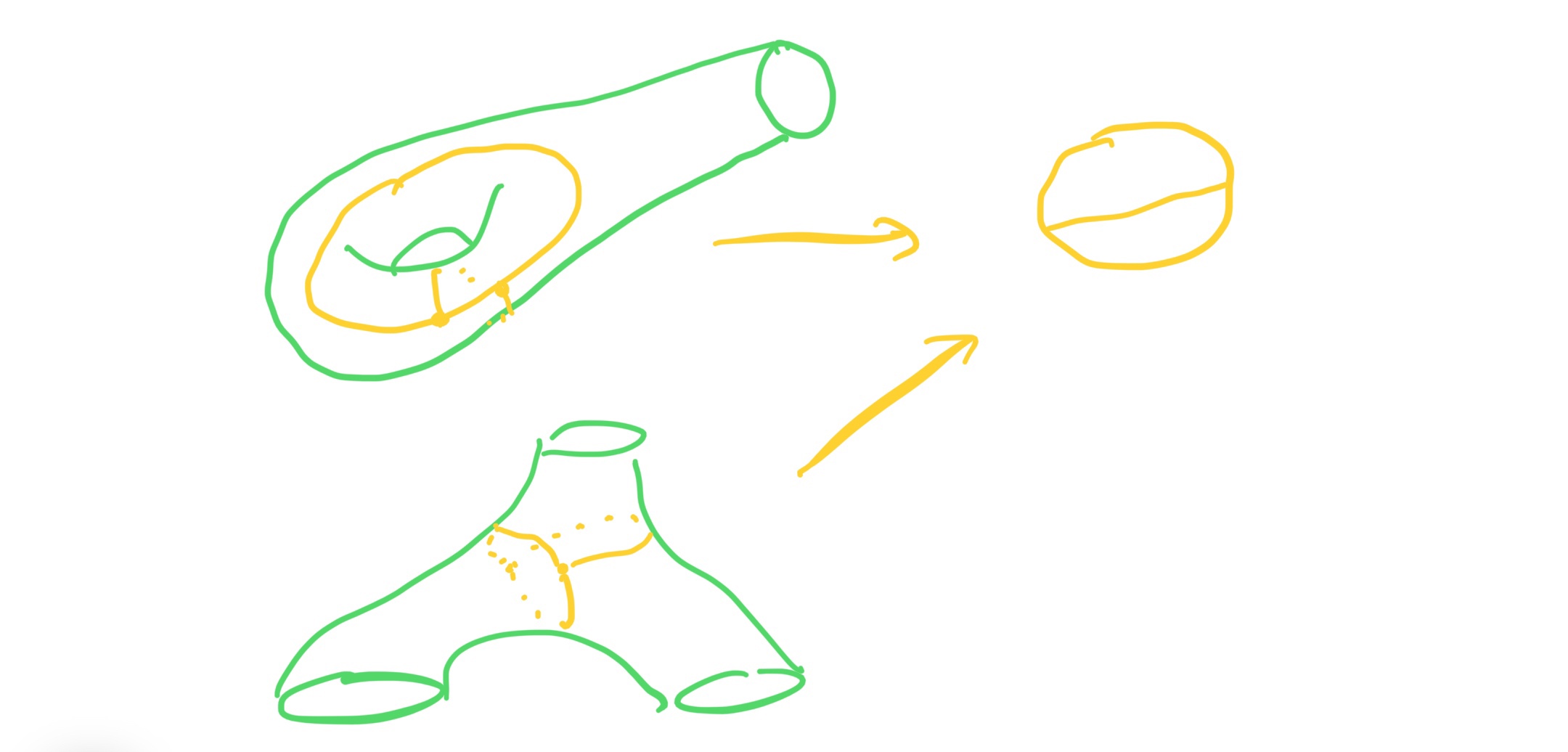}
\end{ex}
Обозначим множество вершин некоторого конечного графа $\Gamma$ за $\mathcal{V}(\Gamma)$, а множество ребер за $\mathcal{E}(\Gamma)$. Для каждой вершины $v\in \mathcal{V}(\Gamma)$ рассмотрим множество инцендентных ему ребер
\begin{equation*}
    \mathcal{E}(\Gamma,v):=\{e\in \mathcal{E}(\Gamma)|\mbox{ ребро } e \mbox{ инцендентно вершине } v\}, \mbox{ } d(v):=\#\mathcal{E}(\Gamma,v).
\end{equation*}
\textbf{Ленточная структура} на графе $\Gamma$ это циклический порядок на каждом из $\mathcal{E}(\Gamma,v)$. На графе, вложенном в ориентированную поверхность, он определен поворотом по часовой стрелке.\\
Ленточному графу $\Gamma$ отвечает компактная ориентированная поверхность с краем $\Sigma(\Gamma)$. 
\begin{mydef}
\textbf{Поверхность} $\Sigma(\Gamma)$ это объединение замкнутых дисков $\mathbb{D}_v$, $v\in \mathcal{V}(\Gamma)$ и лент 
$ R_e:=[0,1]\times [-1,1], \mbox{ } e\in \mathcal{E}(\Gamma)$ склеенных в соответствии с ленточной структурой на $\Gamma$
\begin{equation*}
    \Sigma(\Gamma):=\bigsqcup_{v\in \mathcal{V}(\Gamma)} \mathbb{D}_v\cup \bigsqcup_{e\in \mathcal{E}(\Gamma)} R_e.
\end{equation*} 
\begin{center}
\includegraphics[width=7cm]{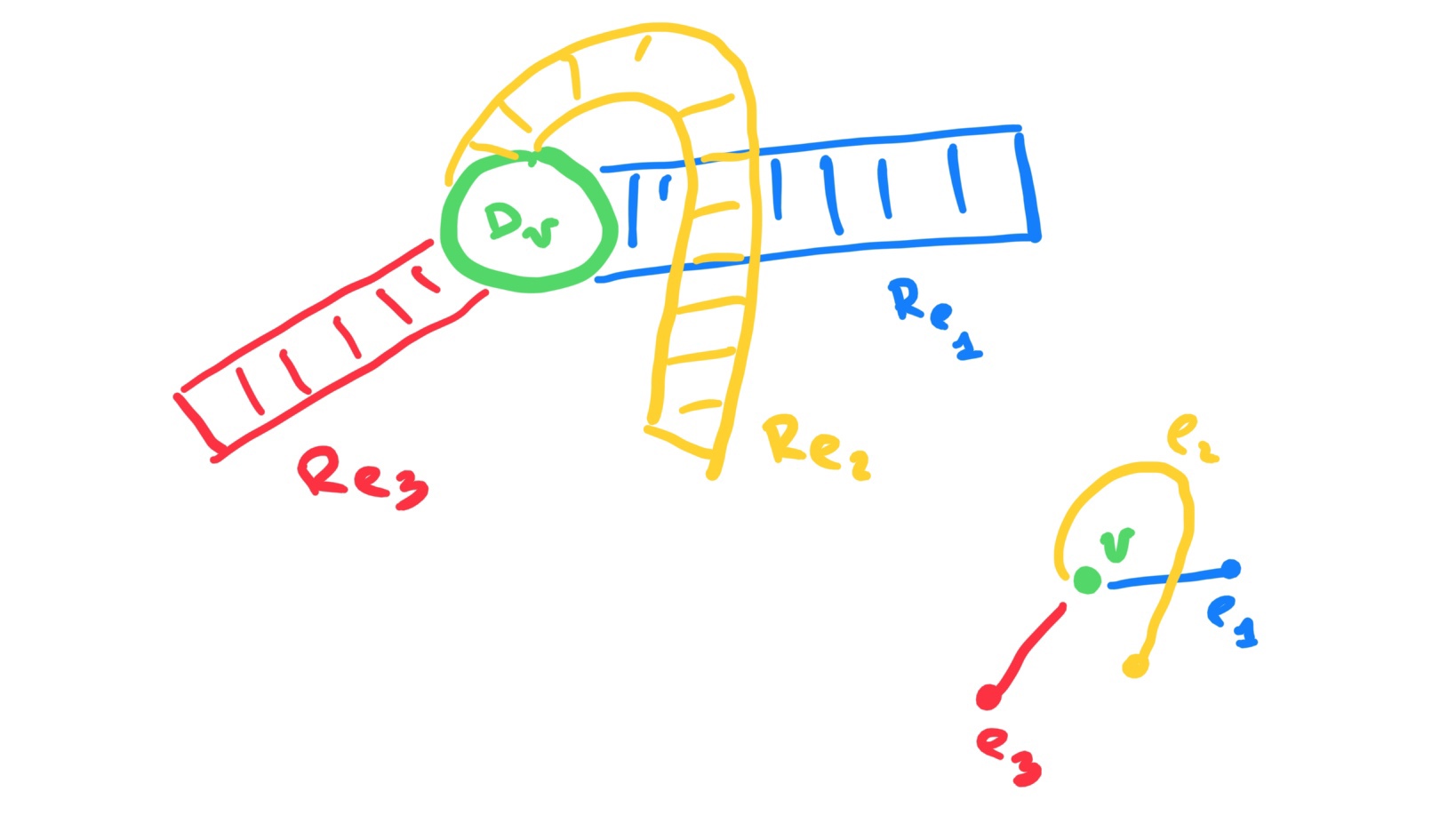}
\newline
Участок поверхности $\Sigma(\Gamma)$, отвечающий вершине и трем инцидентным ей ребрам.
\end{center}
Некомпактная ориентированная поверхность $S(\Gamma):=\widehat{\Sigma(\Gamma)}$ получается пополнением $\Sigma(\Gamma)$
\begin{equation*}
    S(\Gamma):\Sigma(\Gamma)\cup_{\partial (\Sigma(\Gamma))} \partial (\Sigma(\Gamma))\times \mathbb{R}_+. 
\end{equation*}
\end{mydef}
Заметим, что $S(\Gamma)$ гомотопически эквивалентна $\Gamma$. Для графа $\Gamma$, вложенного в ориентированную поверхность,  $S(\Gamma)$ диффеоморфна достаточно малой открытой окрестности $Op(\Gamma)$.\\
Теперь мы построим Лиувиллеву структуру на ${S}(\Gamma)$, с которой будет удобно работать. Можно считать, что валентность $d(v)$ всех вершин $\Gamma$ не превышает трех. Это условие на скелет $\mathbb{S}$ не ограничительно: для Штейновой поверхности этого можно добиться, приведя функцию $f$ в общее положение. В таком случае,  $\Sigma(\Gamma)$ есть результат склейки ленточек и $T$-образных перекрестков, каждый из которых является компактной поверхностью с углами. 
\begin{center}
\includegraphics[width=4cm]{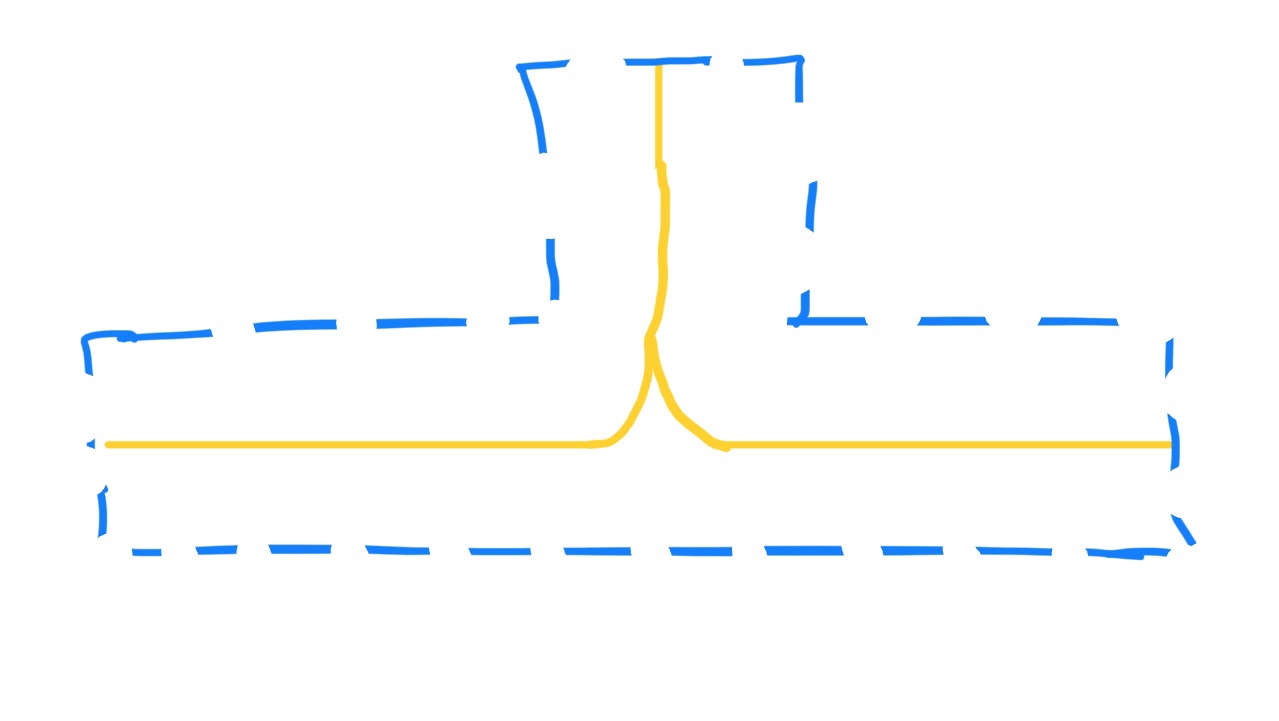}
\end{center}
Для каждой вершины валентности $3$ нам надо зафиксировать одно из инцидентных ей ребер. 
\begin{mydef}
\textbf{Арбореллевский граф} $\mathcal{T}$ это конечный ленточный граф $\Gamma$, для которого
\begin{equation*}
    \forall v\in \mathcal{V}(\Gamma)\mbox{ } d(v)\le 3, 
\end{equation*}
на котором для каждой вершины $v, d(v)=3$ зафиксирована ее \textbf{"ножка"{}}, то есть ребро
\begin{equation*}
   l(v) \in \mathcal{E}(\Gamma,v).
\end{equation*}
Мы предполагаем, что $\Gamma$ не имеет двойных ребер. 
\end{mydef}
Сейчас мы построим Лиувиллеву структуру на поверхности $S(\Gamma)$, используя структуру Арбореллевского графа $\mathcal{T}$. Для каждой вершины  $v\in\mathcal{V}(\Gamma)$ зафиксируем ее окрестность,  пересечение которой с $\Sigma(\Gamma)$ определяется только валентностью $v$ и имеет стандартный вид
\begin{equation*}
    \{|p|^2+|q|^2\le 1\}\cup \{q>0,\mbox{ } |p|<1\} \mbox{ при } d(v)=1,
\end{equation*}
\begin{equation*}
     \{|p|\le 1\} \mbox{ при } d(v)=2 \mbox{ и}
\end{equation*}
\begin{equation*}
    \{|p|\le 1\}\cup \{|q|\le 1, \mbox{ } |p|>0\} \mbox{ при } d(v)=3.
\end{equation*}
\begin{center}
\includegraphics[width=6cm]{8.jpg}
\end{center}
Рассмотрим гладкое вложение $\Gamma$ в $S(\Gamma)$, образ которого содержит объединение "нулевых сечений"{} $J_e:=(0,1)\times \{0\}\subset R_e$ и пересекается с этими открытыми множествами как 
\begin{equation*}
    \{p=0,\mbox{ } q\ge 0\} \mbox{ при } d(v)=1,
\end{equation*}
\begin{equation*}
     \{p=0\} \mbox{ при } d(v)=2 \mbox{ и}
\end{equation*}
\begin{equation*}
     \{p=0\}\cup \{q=0, \mbox{ } p>0\} \mbox{ при } d(v)=3.
\end{equation*}
Каждая из лент $R_e$ определяет в $S(\Gamma)$ подмножество, диффеоморфное $J_e\times \mathbb{R}$. Введем на нем Лиувиллеву структуру, совпадающую со стандартной Лиувиллевой структурой на $T^*[0,1]$. Продолжим ее на каждый из дисков $D_v$ так, что в стандартных координатах на диске симплектическая форма имеет вид $\omega=dp\wedge dq$, Лиувиллева форма отличается от $pdq$ на дифференциал функции $h$, а Лиувиллево векторное поле $X$ имеет следующий вид
\begin{center}
\includegraphics[width=10cm]{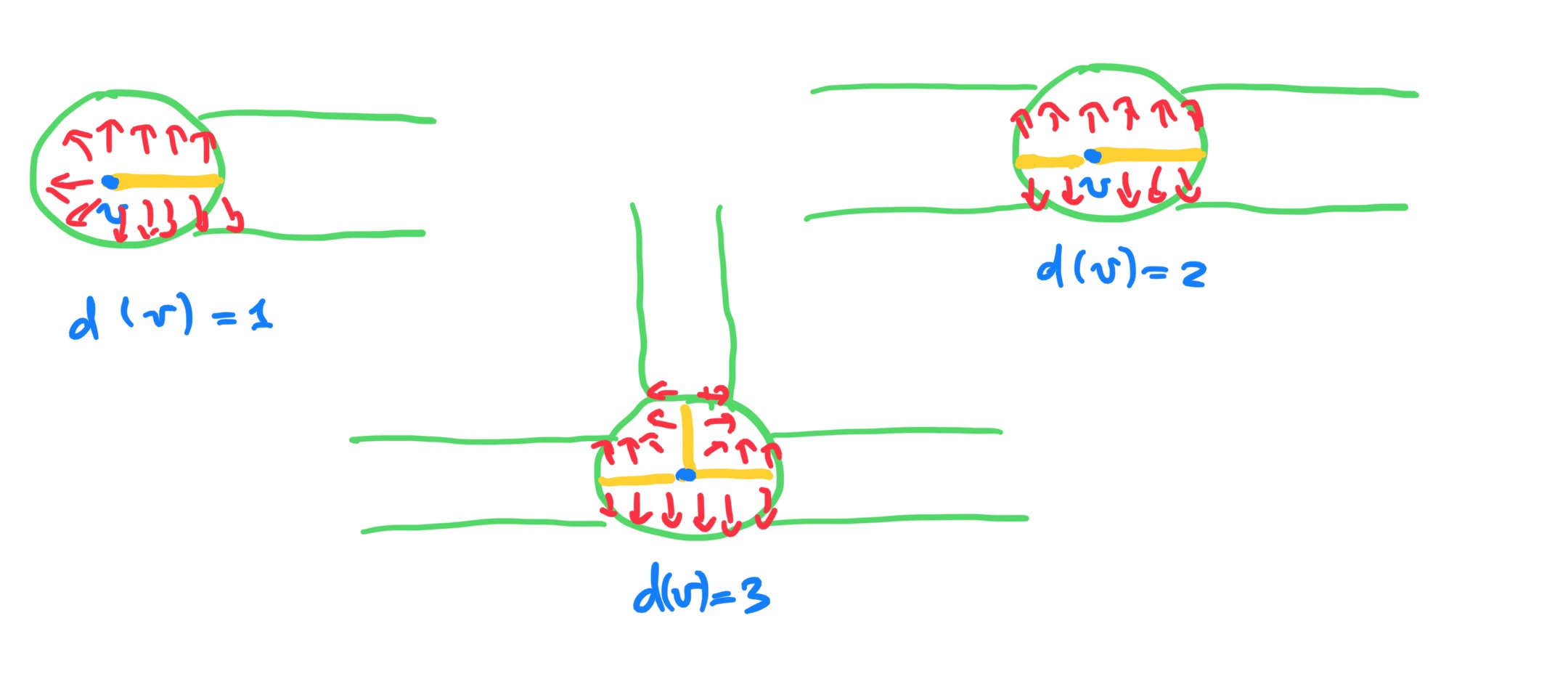}
\end{center}
В качестве $h$ можно выбрать, в зависимости от валентности $d(v)$ вершины $v$, функцию
\begin{itemize}
    \item интерполирующую $0$ на правой полуплоскости и $-pq/2$ на левой, если $d(v)=1$;
    \item тождественно равную нулю, если $d(v)=2$;
    \item инетрполирующую $0$ на дополнении лент шляпки шляпки до окрестности $v$ и $-pq$ на дополнении ленты ножки до окрестности $v$.
\end{itemize} 
\begin{mydef}
\textbf{Лиувиллева поверхность}  $\mathbb{S}(\mathcal{T})$, отвечающая Арбореллевскому графу $\mathcal{T}$ это
\begin{equation*}
    (S(\Gamma), d\lambda, \lambda, X). 
\end{equation*}
\end{mydef}
Заметим, что $\Sigma(\mathcal{T})$ является Лиувиллевой подобластью поверхности $\mathbb{S}(\mathcal{T})$, и кроме того
\begin{equation*}
   Sk(\mathbb{S}(\mathcal{T}))= \Gamma.
\end{equation*}

\section{Точные кривые на Лиувиллевых поверхностях}
\label{sec:curves}
\textbf{Кривой} на Лиувиллевой поверхности $\mathbb{S}$ называется вложение одномерного многообразия
\begin{equation*}
    \iota:C \to S. 
\end{equation*}
Кривая $(C,\iota)$ является Лагранжевым подмногообразием $(S,\omega)$, так как $\omega$ тождественно зануляется на $\iota(C)$. Если к тому же интеграл $\int_{\gamma}\lambda$ равен нулю, кривая называется \textbf{точной}.
\begin{mydef}
\textbf{Точной, или сбалансированной кривой} на $\mathbb{S}$ называется пара
\begin{equation*}
    \gamma=(C,\iota),
\end{equation*}
где $\iota:C \to S$ это параметризованная кривая, для которой существует \textbf{потенциал}
\begin{equation*}
   f\in\mathcal{C}^\infty (C);\mbox{ } df=\iota^*\lambda. 
\end{equation*}
Мы обозначаем той же буквой $\gamma$ соответствующее подмногообразие $\iota(C)$ и пишем $\gamma\subset \mathbb{S}$.
\end{mydef}
Лиувиллев автоморфизм $\phi$ поверхности $\mathbb{S}$ переводит точную кривую $\gamma\subset \mathbb{S}$ в $\phi(\gamma)$
\begin{equation*}
    \phi(\gamma):=(C,\phi\circ\iota,f+h\circ\iota), \mbox{ где}
\end{equation*}
\begin{equation*}
    \gamma=(C,\iota,f), \mbox{ } \phi^*\lambda=\lambda+dh.
\end{equation*}
\begin{center}
\includegraphics[width=5cm]{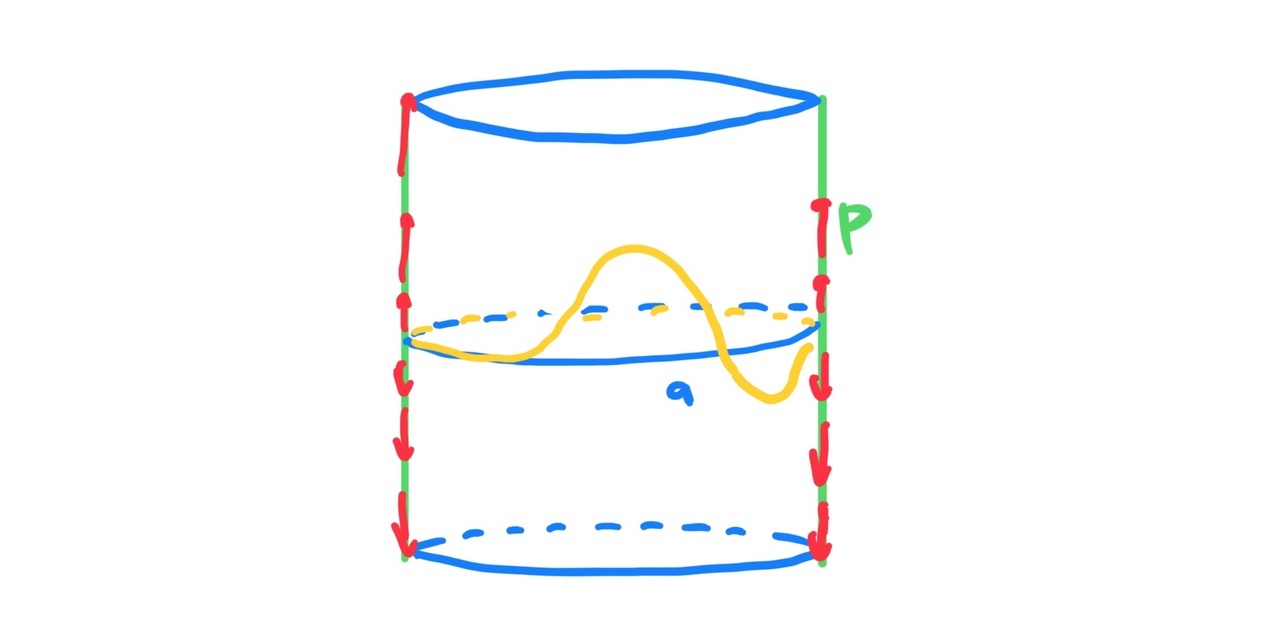}
\end{center}
\textit{Точная кривая, полученная шевелением нулевого сечения цилиндра}.
\subsection*{Производящие семейства на гладких кривых}
Рассмотрим в качестве $(S,\omega)$ кокасательное расслоение $\mathbb{T}^*(Q)$ к окружности или прямой. Существует общая конструкция, позволяющая строить точные кривые на этих поверхностях. 
\begin{mydef}
\label{def:gf}
\textbf{Производящее семейство} над кривой $Q$ это такая пара $(E,F)\in Gf(Q)$ 
\begin{itemize}
    \item $\pi_E:E\to Q$ - тривиализованное гладкое расслоение, слой которого диффеоморфен $\mathbb{R}^n$; 
    \item $F\in \mathcal{C}^\infty(E)$ - функция, для которой $\frac{\partial}{\partial \xi}F(x)$ {трансверсален нулю для каждой точки} $x$
    \begin{equation*}
        F_\xi(x)=d_xF|_{T_x(E_x)}:{T_x(E_x)}\to T_{F(x)}\mathbb{R}\cong \mathbb{R}\ni \{0\}.
    \end{equation*}
\end{itemize}
Множество послойных критических точек семейства $C_{E,F}:=\{F_\xi(x)= 0 \}$ это гладкая кривая на $E$, которую можно погрузить в $\mathbb{T}^*Q$ по формуле ${\iota_{E,F}}(x):= (\pi_E(x), F_q(x))$. Кривая 
\begin{equation*}
    \gamma\subset \mathbb{T}^*Q
\end{equation*} 
\textbf{допускает производящее семейство} $(E,F)\in Gf(Q),$ и \textbf{порождается им}, если существует диффеоморфизм $\phi:C\to C_{E,F}$, для которого выполнено $\iota=\iota_{E,F}\circ \phi,\mbox{ } f=F\circ \phi. $
\end{mydef}
\begin{ex}
Простейший пример производящего семейства задается функцией $f\in\mathcal{C}^\infty(Q)$
\begin{equation*}
    E=Q, \mbox{ } F\equiv f.
\end{equation*}
В таком случае, пара $(E,F)$ называется производящей функцией и определяет вложение 
\begin{equation*}
    C_{E,F}=Q\xrightarrow[]{df} T^*Q.
\end{equation*}
В частности, нулевое сечение $0_Q$ допускает производящую функцию $f\equiv 0$. 
\end{ex}
Малое гамильтоново шевеление точной кривой, допускающей производящее семейство, снова допускает производящее семейство. Обобщение этого утверждения на произвольные гамильтоновы шевеления это главный инструмент в работе с производящими семействами.
\begin{theorem}[Свойство гамильтонова подъема \cite{Chekanov}]
Образ кривой $\gamma_{E,F}\subset \mathbb{T}^*(Q)$ под действием Гамильтоновой изотопии $\psi\in Ham^c(\mathbb{T}^*(Q))$  также допускает производящее семейство. 
\end{theorem}
Производящее семейство для кривой $\gamma_{E,F}\subset \mathbb{T}^*Q$ не единственно. Стартуя с $(E,F)$, мы можем построить новые производящие семейства для $\gamma_{E,F}$, используя следующие операции.
\begin{itemize}
    \item Семейства $(E_i,F_i)\in Gf(Q)$ \textbf{эквивалентны} $(E_1,F_1)\cong (E_2,F_2)$, если существует расслоенный  диффеоморфизм $\psi:E_1\to E_2$, для которого $F_1:=\psi^*F_2$. Тогда $\gamma_{E_i,F_i}$ совпадают
    \begin{equation*}
    \psi(C_{E_1,F_1})=C_{E_2,F_2}, \mbox{ } \iota_{E_1,F_1}=\iota_{E_1,F_2}\circ \psi.
    \end{equation*}
    \item Семейство $(E'.F')$ называется \textbf{стабилизацией} семейства $(E,F)$, если для некоторого векторного пространства $V$ и невырожденной квадратичной формы $q$ на $V$ выполнено
    \begin{equation*}
     (E',F'):=(E,F)\times (V,q).    
    \end{equation*}
     Семейство $(E',F')$ порождает кривую $\gamma_{E,F}$, так как $C_{E',F'}=C_{E,F}\times (0), \mbox{ } \iota_{E',F'}=\iota_{E,F}.$
\end{itemize}
Комбинируя эти две операции, получаем следующее определение 
\begin{mydef}
Производящие семейства $(E_i,F_i)\in Gf(Q)$ \textbf{стабильно эквивалентны}
\begin{equation*}
    (E_1,F_1)\cong^{st} (E_2,F_2),
\end{equation*}
если существует невырожденная квадратичная форма $q$ на пространстве $\mathbb{R}^n$, для которой
\begin{equation*}
    (E_1,F_1)\oplus (\mathbb{R}^n,q)\cong (E_2,F_2). 
\end{equation*}
\end{mydef}
\subsection*{Стандартное покрытие поверхности $\mathbb{S}(\mathcal{T})$}
Зафиксируем Арбореллевский граф $\mathcal{T}$. Сейчас мы построим покрытие соответствующей ему Лиувиллевой поверхности $\mathbb{S}(\mathcal{T})$ картами Дарбу, пронумерованными вершинами графа $\mathcal{T}$
\begin{equation*}
    U_{v},\mbox{ } v\in \mathcal{V}(\Gamma).
\end{equation*}
На $\Sigma(\Gamma)$ задано покрытие $\{W_v\}$, пронумерованное  множеством вершин графа $\mathcal{V}(\Gamma)$
\begin{equation*}
    \Sigma(\Gamma) = \bigcup_{v\in\mathcal{V}(\Gamma)}W_{v}, \mbox{ } W_v:=\mathbb{D}\cup \bigsqcup_{e\in \mathcal{E}(\Gamma,v)} R_e.
\end{equation*} 
В случае валентности $d(v)=3$ карта $W_{v}$ отождествляется со следующим подмножеством
\begin{equation*}
    \Sigma(\boldsymbol{\perp})=\{(q,p)\in\mathbb{V}|\mbox{ где }  p\ge -1 \mbox{ и } p\le 1 \mbox{ при } |q|\ge 1\}
\end{equation*}
которое называется \textbf{перекрестком}. Мы обозначаем за $\boldsymbol{\perp}$ граф с $1$ вершиной и $3$ полуребрами.
\begin{center}
\includegraphics[width=4cm]{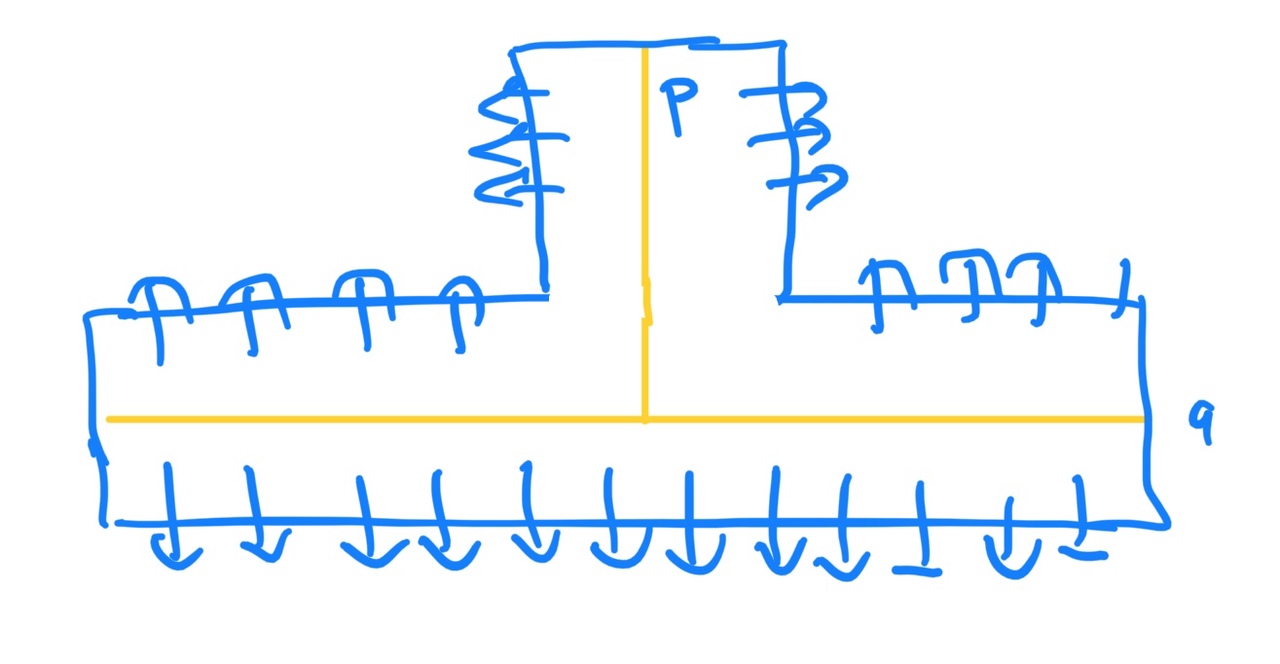}
\end{center}
Обозначим за $I_e\subset S(\Gamma)$ интервал ребра $e$. Для вершины $v\in\mathcal{V}(\Gamma)$ рассмотрим интервал $I_v$, полученный пересечением объединения $\{v\}\cup \bigcup_{e\in \mathcal{E}(\Gamma,v)-\{l(v)\}} I_e$ и карты $W_v$
\begin{equation*}
   I_v:= (\{v\}\cup \bigcup_{e\in \mathcal{E}(\Gamma,v)-\{l(v)\}} I_e)\cap W_v\subset S(\Gamma).
\end{equation*}
Каждый из интервалов $I_v$ гладко вложен в $\mathbb{S}(\mathcal{T})$. Для $d(v)=3$ он называется \textbf{"шляпкой"{}}.\\
Пересечение $W_{v}$ и $W_{w}$ есть объединение лент, отвечающих ребрам, соединяющим $v$ и $w$  
\begin{equation*}
    W_{v}\cap W_{w}=\bigcup_{e\in \mathcal{E}(\Gamma,v,w)} R_e,
\end{equation*}
\begin{equation*}
    \mathcal{E}(\Gamma,v,w):=\mathcal{E}(\Gamma,v)\cap \mathcal{E}(\Gamma,w).
\end{equation*}
\begin{ex}
На рисунках изображены покрытия картами "штанов"{} и проколотого тора.
\begin{center}
\includegraphics[width=6cm]{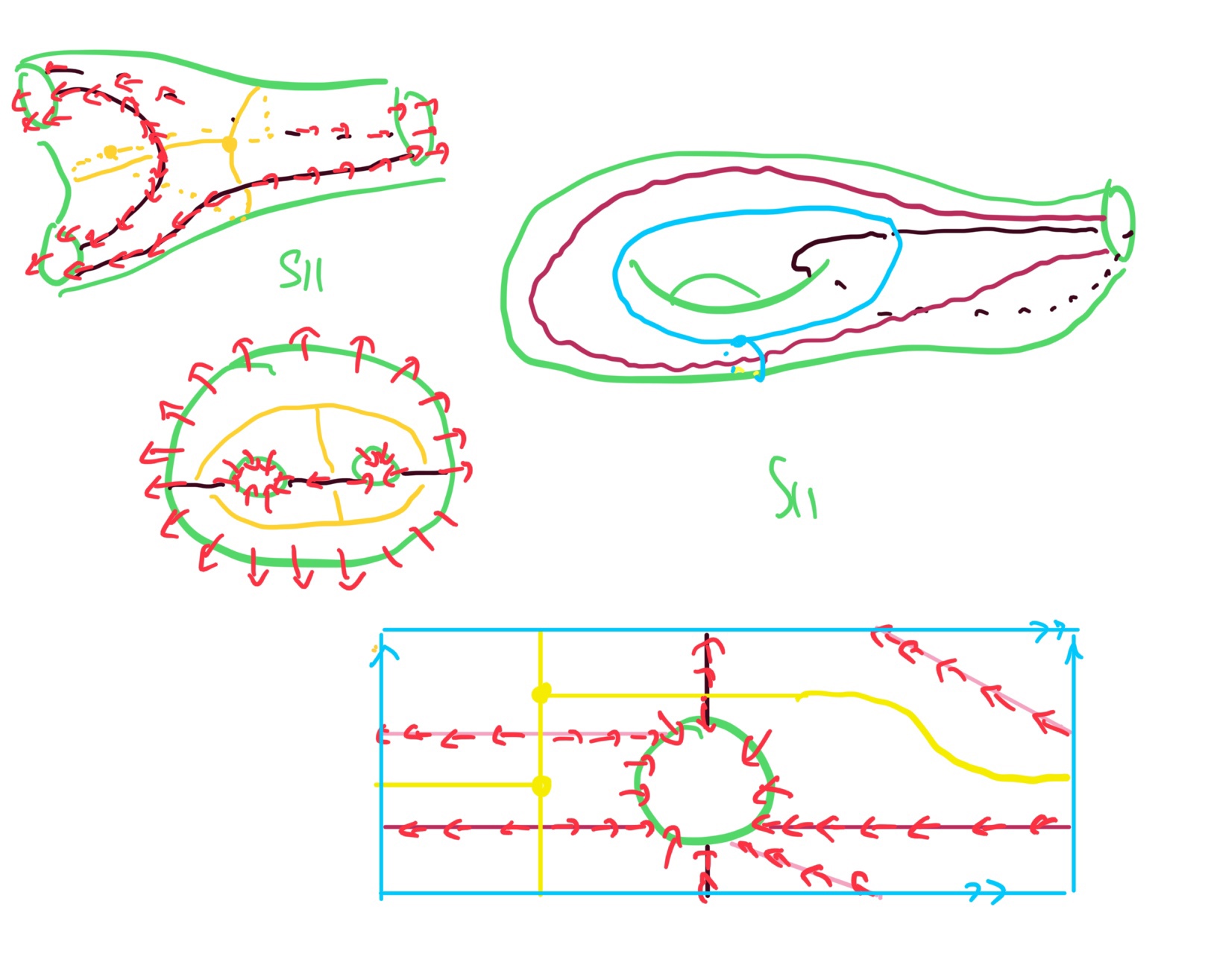}
\end{center}
\end{ex}
Покрытие $W_v$ определяет покрытие $S(\Gamma)$ координатными картами ${U_v}$, каждая из которых в свою очередь представляется как объединение открытого диска и нескольких копий $T^*[0,1]$.\\
Каждая из карт $U_v$ допускает симплектоморфизм на $\mathbb{V}:=(\mathbb{R}^2_{q,p},dp\wedge dq)$. который переводит 
\begin{equation*}
    I_e \mbox{ в } O_q:=\{p=0\},\mbox{ } \mathbb{D}_v \mbox{ в } \{p^2+q^2<1\}.
\end{equation*}
Таким образом, $W_v$ симплектоморфно $\mathbb{T}^*(I_e)$, так что $U_{v}$ являются картами Дарбу поверхности $\mathbb{S}(\mathcal{T})$. Для каждого $e\in\mathcal{E}(\Gamma)$ задано вложение $ T^*I_e\to\mathbb{S}$, продолжающее вложение $R_e\to \mathbb{S}$ и являющееся симплектоморфизмом на свой образ.
\begin{center}
\includegraphics[width=4cm]{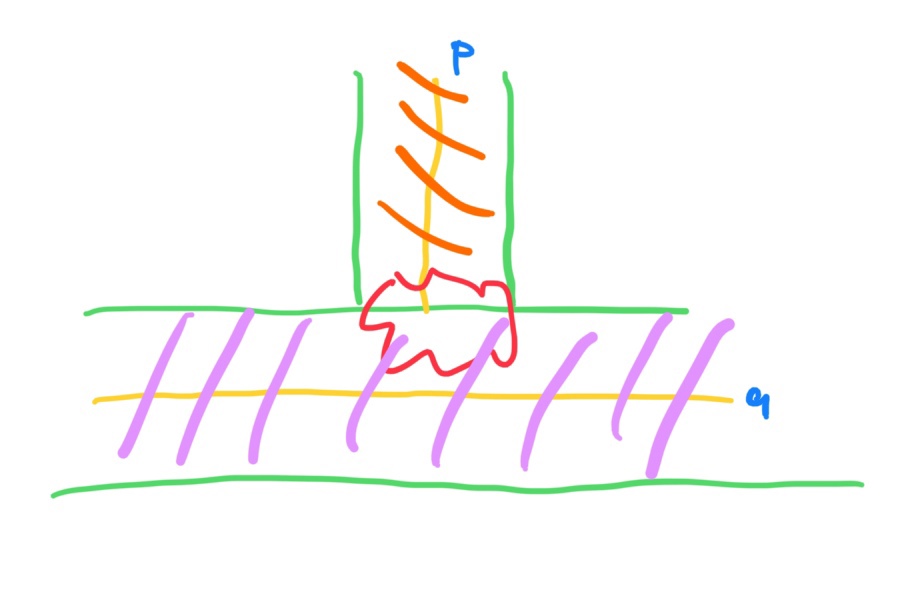}
\end{center}
Поток векторного поля $-X$ определяет ретракцию $\pi$ поверхности $\mathbb{S}(\Gamma)$ на граф $\Gamma$. \\
Каждая из карт $U_v$ инвариантна относительно этой проекции, то есть $\pi$ склеивается из
\begin{equation*}
   \pi_v: U_{v}\to \bigcup_{e\in \mathcal{E}(\Gamma,v)} I_e.
\end{equation*}
На пересечениях карт $\pi$ совпадает с проекцией касательного расслоения
\begin{equation*}
    T^*(J_e)\to J_e.
\end{equation*}
\subsection*{Кривые, локально допускающие производящее семейство}
Рассмотрим точную кривую $\gamma=(C,\iota,f)$ на $\mathbb{S}(\mathcal{T})$. Обозначим ее пересечение с картами $U_v$ как
\begin{equation*}
    \gamma_v=(C_v,\iota_v,f_v),
\end{equation*}
\begin{equation*}
    C_v:=\iota^{-1}(U_v), \mbox{ } \iota_v:=\iota|_{C_v},\mbox{ } f_v:=f|_{C_v}.
\end{equation*}
Каждая из карт $U_v$ снабжена выделенным симплектоморфизмом на $\mathbb{T}^*(I_e)$. Таким образом, производящее семейство $(E_e,F_e)$ на ребре $I_e$ порождает точную кривую $\gamma_{E_e,F_e}$ на $U_v$.
\begin{mydef}
Кривая $\gamma\subset \mathbb{S}(\mathcal{T})$ \textbf{локально допускает производящие семейства}, если
$\forall v\in \mathcal{V}(\Gamma)$ точная кривая $\gamma_v\subset \mathbb{T}^*(I_v)$, допускает производящее семейство $(E_v,F_v)\in Gf(I_v)$.
\end{mydef}
Это определение выглядит не слишком осмысленно.
Мы получим более содержательное объект, если для каждого ребра $e\in  \mathcal{E}(\Gamma,v,w)$ наложим условие согласованности на
\begin{equation*}
    (E_v,F_v) \mbox{ и } (E_w,F_w). 
\end{equation*}
Предположим для начала, что ребро $e$ не является ножкой ни для $v$, ни для $w$. Тогда семейства можно ограничить на $J_e$. В качестве условия согласованности можно потребовать, чтобы 
\begin{equation*}
  (E_w,F_w)|_{J_e}\cong^{st}  (E_v,F_v)|_{J_e}. 
 \end{equation*}
Чтобы наложить аналогичное условие в случае, когда $e$ является ножкой для одной из вершин, которые она соединяет, надо для каждой вершины $w$, $d(w)=3$ определить ограничение
\begin{equation*}
    \cdot |_{l(w)}:Gf(I_w)\to Gf(I_{l(w)}),
\end{equation*}
такое, что для каждого $(E,F)\in Gf(I_w)$ производящее семейство $(E,F)|_{l(w)}$ порождает
\begin{equation*}
    \gamma_{l(e)}=(C_{l(e)},\iota_{l(e)},F_{l(e)}),
\end{equation*}
\begin{equation*}
    C_{l(e)}:=\iota_{E,F}^{-1}(T^*(I_{l(e)})), \mbox{ } \iota_{l(e)}:=\iota_{E,F}|_{C_{l(e)}},\mbox{ } F_{l(e)}:=F|_{C_{l(e)}}.
\end{equation*}
Так как покрытие имеет стандартный вид, достаточно построить $\cdot |_{l}$ для перекрестка  $\mathbb{S}(\boldsymbol{\perp})$.\\
Перейдем от $\mathbb{S}(\boldsymbol{\perp})$ к плоскости $\mathbb{V}=(\mathbb{R}^2_{q,p},dp\wedge dq)$, которую можно рассматривать как кокасательное пространство к любой из двух координатных осей. Таким образом, производящему семейству на $O_q$ (или на $O_p$) отвечает точная кривая на $\mathbb{V}$. Мы докажем такое утверждение
\begin{prop}
\label{prop:turn}
Кривая $\gamma\subset\mathbb{V}$ допускает производящие семейства над $O_q$ и $O_p$ одновременно.
\end{prop}
Для этого нам надо научиться переписывать семейство над $O_q$ как семейство над $O_p$, то есть семейству $(E,F)\in Gf(O_q)$ сопоставить $(\widetilde{E},\widetilde{F})\in Gf(O_p)$ так, чтобы $ \gamma_{\widetilde{E},\widetilde{F}}=\gamma_{E,F}.$ Тогда
\begin{equation*}
    (E,F)|_{l(w)}:=(\widetilde{E},\widetilde{F})|_{p>0}.
\end{equation*}
и имеет смысл такое же условие согласованности, что и было наложено выше. \\
Рассмотрим симплектоморфизм $W$, поворачивающий $\mathbb{V}$ на угол $\pi$ по часовой стрелке
\begin{equation*}
    W(q,p):=(p,-q)
\end{equation*}
Зафиксируем отождествление $\psi:O_q\to O_p$, заданное ограничением симплектоморфизма $W$.\\
В терминах обратного образа $ ({E}^{\circlearrowright},{F}^{\circlearrowright}):=\psi^*(\widetilde{E},\widetilde{F})\in  Gf(O_q)$ получаем следующее условие
\begin{equation*}
  \gamma_{\widetilde{E},\widetilde{F}}=\gamma_{E,F}\Leftrightarrow  \gamma_{{E}^{\circlearrowright},{F}^{\circlearrowright}}= W(\gamma_{E,F}). 
\end{equation*}
Перед тем, как привести в разделе~\ref{sec:main} формулу для $\circlearrowright$, мы хотим описать, как ее получить. Наша конструкция работает в более широкой ситуации. Мы сделаем отступление и разберем в разделе~\ref{sec:chekanov} общий случай, чтобы мотивировать определение $\circlearrowright$ и понять его свойства.

\section{Оператор Чеканова}
\label{sec:chekanov}
\subsection*{Производящие семейства для симплектоморфизмов}
Изложенная ниже конструкция работает для произвольного симплектоморфизма $S:\mathbb{V}\to \mathbb{V}$ \textbf{допускающего производящее семейство}. Сейчас мы сформулируем это условие. \\
Отождествим $O_q$ с Евклидовым пространством
${L}=(\mathbb{R},g)$, а $V$ с прямой суммой
\begin{equation*}
    V=(L\oplus L,\langle\cdot,\cdot\rangle),
\end{equation*}
\begin{equation*}
    \langle (u,l),(s,v) \rangle = g(u,s)+g(l,v).
\end{equation*}
Подстановка в спаривание $\langle\cdot,\cdot\rangle$ определяет изоморфизм $V\leftrightarrow V^*$, применение которого обозначается добавлением $\tilde{\cdot}$. Тогда определен симплектоморфизм $\varpi_{\mathbb{V}}:\overline{\mathbb{V}} \times \mathbb{V}\xrightarrow{} \mathbb{T}^*V$
\begin{equation*}
       \varpi_{\mathbb{V}} (u,l,s,v)=(q_1,q_2,p_1,p_2):=(\frac{u+s}{2}, \frac{{l}+{v}}{2}, \widetilde{v}-\widetilde{l}, \widetilde{u}-\widetilde{s}).
\end{equation*}
Рассмотрим образ графика $gr_S:=\{(v,S(v))\}$ симплектоморфизма $S$ под действием $\varpi_{\mathbb{V}}$
\begin{equation*}
    \mathcal{L}_S:=\varpi_{\mathbb{V}}(gr_{S})\subset \mathbb{T}^*V.
\end{equation*}
Заметим, что $\mathcal{L}_S$ является Лагранжевым подмногообразием кокасательного расслоения $\mathbb{T}^*V$. 
\begin{mydef}
\textbf{Производящее семейство} для $S$ это такое $(H,G)\in Gf(V)$, что
\begin{equation*}
    \mathcal{L}_S=\{(q,p)|\mbox{ }\exists x\in H_q\mbox{, } \frac{\partial}{\partial q}G(x)= p,\mbox{ } \frac{\partial}{\partial \xi}G(x)= 0\}.
\end{equation*}
\end{mydef}
Рассмотрим следующие данные:
\begin{itemize}
    \item производящее семейство $(E,F)\in Gf(L)$;
    \item производящее семейство $(H,G)\in Gf(V)$, порождающее симплектоморфизм $S$.
\end{itemize}
Обозначим за $M(E;S)$ расслоение $E\times H\to L$. Подставим $E= L_q\times W_{\xi}$ и $H=V_{v,t}\times U_\chi$ 
\begin{equation*}
    M(E;S)=L_q\times W_{\xi}\times  L_v\times L_t\times U_{\chi}.
\end{equation*}
Определим функцию $P(F,S)\in\mathcal{C}^\infty(M(E;S))$, заданную следующим образом
\begin{equation*}
    P(F;S):=F(u,\xi)+G(l,t,\chi)-g(v,t)
    \mbox{, где}
\end{equation*}
\begin{equation*}
    u=q+v,\mbox{ } l=q+v/2.
\end{equation*}
\begin{mydef}
\textbf{Оператор Чеканова} для симплектоморфизма $S$ с фиксированным производящим семейством $(H,G)\in Gf(V)$ это следующее отображение $Ch_S:Gf(L)\to Gf(L)$ 
\begin{equation*}
    Ch_S(E,F)=(M(E;S), P(F;S))
\end{equation*}
\end{mydef}
\begin{prop}
Для стабильно эквивалентных семейств $(E_i,F_i)\in Gf(L)$ выполнено
\begin{equation*}
Ch_S(E_1,F_1)\cong^{st}Ch_S(E_2,F_2).
\end{equation*}
\begin{proof}
1) Предположим, что $(E_i,F_i)$ эквивалентны, то есть существует расслоенный диффеоморфизм, $\psi:E_1\to E_2$ для которого $F_1=\psi^*F_2$. Тогда $P_1=\eta^*P_2$ для  
    \begin{equation*}
        \eta:M_1\to M_2, \mbox{ } \eta(q,\xi,v,t,\chi):=(q,\psi(\xi),v,t,\chi)
    \end{equation*}
2) Если $(E'.F')$ это {стабилизация} семейства $(E,F)$, то $(M',P')$ является стабилизацией $(M,P)$, так как для векторного пространства $V$ и невырожденной квадратичной формы $q$ на $V$ верно
    \begin{equation*}
       Ch_S[(E,F)\oplus (V,Q)] = Ch_S(E,F)\oplus (V,Q),
    \end{equation*}
\end{proof}
\end{prop}
\begin{ex}
Заметим, что $\varpi_{\mathbb{V}}$ переводит диагональ $\Delta_V$ в $0_V$, так что тождественный симплектоморфизм $Id_V$ допускает производящую функцию $G\equiv 0$. Применим оператор $Ch_{Id}$
\begin{equation*}
    P(F;Id_V):=F(q+v,\xi)-g(v,t).
\end{equation*}
Легко видеть, что $Ch_{Id_V}(E,F)$ порождает кривую $\gamma_{E,F}$. Заметим, что более того 
\begin{equation*}
    (E,F) \cong^{st} Ch_{Id_V}(E,F). 
\end{equation*} 
Координата $q$ определяет отождествление прямой $L$ с множеством вещественных чисел. Так как разность $f(q,\xi,v):=F(q+v,\xi)-F(q,\xi)$ тождественно зануляется при $v=0$, корректно определено следующее гладкое отображение $\psi$ из расслоения $M$ в себя
\begin{equation*}
    \psi: (q,\xi,v,t)\mapsto (q,\xi,v,t+\frac{f(q,\xi,v)}{v}).
\end{equation*}
Очевидно, это расслоенный диффеоморфизм, для которого верно $\psi^*P(F,Id_V)=F.$
\end{ex}
Оператор $Ch_S$ был придуман с использованием описания действия симплектической редукции на Лагранжевых подмногообразиях, допускающих производящие семейства.  Аналогичная процедура впервые была применена в статье~\cite{Chekanov}. 
\subsection*{Главное свойство оператора $Ch$}
Зафиксируем $(E,F)$, $S$ и $(H,G)$. Обозначим за $\phi$ ограничение на $C_{M,P}$ отображения 
\begin{equation*}
    M\to E, \mbox{ } (q,\xi,v,t,\chi)\mapsto (q+v,\xi)
\end{equation*}
\begin{prop}
\label{prop:chekanov}
$\phi$ является диффеоморфизмом на $C_{E,F}$ и определяет равенство кривых
\begin{equation*}
 \iota_{M,P}= S\circ\iota_{E,F}\circ \phi\Rightarrow \gamma_{M,P}=S(\gamma_{E,F}).
\end{equation*}
\begin{proof}
Множество  $C_{M,P}$ послойных критических точек $P$ задается как
\begin{equation*}
   Z:= (q,\xi,v,t,\chi)\in C_{M,P}\Leftrightarrow \begin{cases}
      F_\xi (u,\xi)=0,\\
      G_{\chi} (l,t,\chi)=0,\\
      P_{v}(Z)=0,\\
      P_{t}(Z)=0,
    \end{cases}, 
    \mbox{ где } 
    \begin{cases}
      u=q+v,\\
      l=q+v/2.
    \end{cases}
\end{equation*}
Последние два уравнения образуют следующую систему уравнений на переменные $q$, $v$ и $t$
\begin{equation*}
    \centering
    \begin{cases}
      P_{v}(Z)=0,\\
      P_{t}(Z)=0
    \end{cases} 
    \Leftrightarrow
    \begin{cases}
     F_q(u,\xi)+\frac{1}{2}G_{v}(l,t,\chi)-\tilde{t}=0,\\
     G_{t}(l,t,\chi)-\tilde{v}=0,
    \end{cases}
    \Leftrightarrow
    \begin{cases}
      G_{v}(l,t,\chi)=2\tilde{t}-2F_q(u,\xi),\\
      G_{t}(l,t,\chi)=\tilde{v}.
    \end{cases}
\end{equation*}
Сопоставим четверке $(q,\xi,v,t)\in L\times W\times  L\times L$ кокасательный вектор 
\begin{equation*}
    T:=(q+v/2, t,2\tilde{t}-2F_q(u,\xi),\tilde{v})\in T^*V.
\end{equation*}
Так как $(H,G)$ порождает Лагранжево подмногообразие $\mathcal{L}_S$, следующая система на $\chi\in U$
\begin{equation*}
    \begin{cases}
      G_{\chi} (l,t,\chi)=0,\\
      P_{v}(Z)=0,\\
      P_{t}(Z)=0,
    \end{cases}
    \Leftrightarrow
    \begin{cases}
      G_{\chi} (l,t,\chi)=0,\\
      G_{v}(l,t,\chi)=2\tilde{t}-2F_q(u,\xi),\\
      G_{t}(l,t,\chi)=\tilde{v}
    \end{cases}
\end{equation*}
имеет решение тогда и только тогда, когда $T\in \mathcal{L}_S$. В этом случае, решение $\chi$ единственно. \\
Пользуясь равенством $\mathcal{L}_S:=\varpi_{\mathbb{V}}(gr_{S})$ и формулой для обращения симплектоморфизма $\varpi_{\mathbb{V}}$
\begin{equation*}
    \varpi_{\mathbb{V}}^{-1}(q_1,q_2,p_1,p_2)=(q_1+\tilde{p}_2/2,q_2-\tilde{p}_1/2,q_1-\tilde{p}_2/2,\tilde{p}_1/2+q_2),
\end{equation*}
переписываем это условие в терминах $\varpi_{\mathbb{V}}^{-1}(T)=(q+v,\widetilde{F_q(u,\xi)},q,2t-\widetilde{F_q(u,\xi)})$ как
\begin{equation*}
    T\in \mathcal{L}_S\Leftrightarrow\varpi_{\mathbb{V}}^{-1}(T)\in gr_{S}\Leftrightarrow
\end{equation*}
\begin{equation*}
   \Leftrightarrow S(q+v,\widetilde{F_q(u,\xi)})=(q,2t-\widetilde{F_q(u,\xi)}).
\end{equation*}
Пользуясь выражением $P_q(Z)=F_q(q+v,\xi)+G_v(q+v/2,t,\chi)$ получаем равносильность
\begin{equation*}
     T\in \mathcal{L}_S\Leftrightarrow S(u,\widetilde{F_q(u,\xi)})=(q,\widetilde{P_q(Z)}).
\end{equation*}
Таким образом, проекция вдоль $L\times U$ определяет диффеоморфизм $C_{M,P}$ и подмногообразия
\begin{equation*}
G:=\{(q,\xi,v)|\mbox{ } F_\xi (u,\xi)=0,\mbox{ }  S(u,\widetilde{F_q(u,\xi)})=(q,\widetilde{P_q(Z)}) \mbox{ где }u=q+v\}\subset L\times W\times  L.
\end{equation*}
Ограничение отображения $L\times W\times  L\to E$, $(q,\xi,v)\to (q+v,\xi)$ на $G$ бьет в $C_{E,F}$. Оно является диффеоморфизмом, так как  допускает гладкое обратное отображение, заданное 
\begin{equation*}
    (u,\xi)\mapsto (q,\xi,u-q), \mbox{ где } q:=pr_L\circ S\circ \iota_{E,F}(u,\xi).
\end{equation*}
Таким образом $\phi$ является композицией диффеоморфизмов. Осталось заметить, что
\begin{equation*}
    \iota_{M,P}(q,\xi,v,t,\chi)=(q,\widetilde{P_q(Z)})=S(u,\widetilde{F_q(u,\xi)})=S(\iota_{E,F}(u,\xi))= S\circ\iota_{E,F}\circ \phi (q,\xi,v,t,\chi).
\end{equation*}
\end{proof}
\end{prop}
\subsection*{Локализация  оператора Чеканова}
Оператор Чеканова $Ch_S$ не обладает важным свойством локальности: при малом изменении $F$, которое может вообще не поменять $C_{E,F}$ и $\iota_{E,F}$, функция $P$ меняется драматически.  
\begin{ex}
Рассмотрим производящее семейство $(E,F)\in Gf(L)$ и разложение
\begin{equation*}
    F=f+\epsilon,
\end{equation*} 
для которого функция $\epsilon\in\mathcal{C}^\infty(E)$ имеет компактный носитель. Тогда точная кривая $\gamma_{E,F}$ совпадает с подмножеством $\gamma_{E,f}\subset V$ вне компакта, а значит это верно и для их образов
\begin{equation*}
    S(\gamma_{E,F})\cap (V-K) = S(\gamma_{E,f}) \cap (V-K).
\end{equation*}
Тем не менее, ограничение на слой $M_q$ разности 
\begin{equation*}
 \delta|_{M_q}:=(P(F)-P(f))|_{M_q}=\epsilon(u,\xi)   
\end{equation*}
не зануляется тождественно ни при каких значениях параметра $q$.
\end{ex}
До конца этого раздела мы для каждого такого $\epsilon$ построим расслоенный диффеоморфизм $\psi:M\to M$, после применения которого носитель $\delta$ становится компактен.\\
Мы будем использовать утверждение, которое сейчас докажем, в более общей ситуации. Зафиксируем произвольное разложение $F=f+\epsilon$ без условий на $f$ и $\epsilon$. Положим, как раньше 
\begin{equation*}
    (M,P):=(M(E,W),P(F,W)).
\end{equation*}
Зафиксируем произвольную функцию $\phi\in\mathcal{C}^\infty(L)$ и определим новое $\widehat{P}\in\mathcal{C}^\infty(M)$ как
\begin{equation*}
    \widehat{P}(q,\xi,v,t,\chi)=f(u,\xi)+\phi(q)\epsilon(u,\xi)+G(l,t,\chi)-g(v,t)
    \mbox{, где}
\end{equation*}
\begin{equation*}
    u=q+v,\mbox{ } l=q+v/2.
\end{equation*}
Естественно, для $\phi\equiv 1$ мы получаем предыдущую функцию $P$. С другой стороны, если носители $\phi$ и $\epsilon$ компактны, $\widehat{P}$ совпадает с $P(f)$ вне компакта. Сейчас мы сформулируем условия, при которых $(M,\widehat{P})$ порождает $\gamma_{M,P}$. Для начала, введем для каждого $r\in L$
\begin{equation*}
    F_{r}:= f+\phi(r)\epsilon \in \mathcal{C}^\infty (M). 
\end{equation*}
Теперь рассмотрим подмножество $O\subset L$, определенное следующим образом 
\begin{equation*}
    O:=pr_L\circ S\circ\iota_{E,F}(C_{E,F}\cap supp\mbox{ } \epsilon)\cup \{r\in L|\mbox{ } \exists (q,\xi)\in C_{E,F_r}\cap supp\mbox{ } \epsilon\mbox{ : } pr_L\circ S\circ\iota_{E,F_r}(q,\xi)=r\}. 
\end{equation*}
\begin{prop}
\label{prop:localization}
Если $\phi\equiv 1$ на некоторой открытой окрестности $Op(O)$, то $\gamma_{M,P}=\gamma_{M,\widehat{P}}.$
\begin{proof}
Мы докажем, что $P$ и $\widehat{P}$ тождественно совпадают на открытой окрестности 
\begin{equation*}
    C_{M,P}\cup  C_{M,\widehat{P}}\subset M.
\end{equation*}
Из этого сразу следует, что $C_{M,P}=C_{M,\widehat{P}}$ и $\iota_{M,P}\equiv\iota_{M,\widehat{P}}$, а значит и равенство кривых. \\
1) Заметим, что $\forall r\in L$ функции $F$ и $F_{r}$ тождественно равны на окрестности $C_r\cup \widehat{C}_r$, где
\begin{equation*}
    C_r=\{(q,\xi)\in C_{E,F}|\mbox{ }pr_L\circ S\circ\iota_{E,F}(q,\xi)=r\}, \mbox{ } \widehat{C}_r:= \{(q,\xi)\in C_{E,F_r}|\mbox{ } pr_L\circ S\circ\iota_{E,F_r}(q,\xi)=r\}. 
\end{equation*}
Действительно $F-F_r=(1-\phi)\epsilon$, $\phi|_{Op(O)}\equiv 1$, а для $r\notin O$ по определению $C_r\cup \widehat{C}_r \subset E-supp\mbox{ }\epsilon.$\\
Тогда $C_r=\widehat{C}_r$ и кроме того тождественно равны $pr_{L^*}\circ S\circ\iota_{E,F}|_{C_r}=pr_{L^*}\circ S\circ\iota_{E,F_r}|_{C_r}.$\\
2) Для $Q\in \mathcal{C}^\infty(M)$ обозначим множество критических точек ограничения $Q$ на слой ${M_{q_0}}$ за
\begin{equation*}
    X_{q_0}(Q):=C_{M,Q}\cap M_{q_0}=\{(q_0,\xi,l,v,\chi)|\mbox{ } dQ|_{ W\times  L\times L\times U}=0\}\subset M_{q_0}.
\end{equation*}
Из доказательства утверждения \ref{prop:chekanov} следует, что для каждой функции $H\in\mathcal{C}^\infty(E)$ выполнено
\begin{equation*}
    X_{q_0}(P(H))=\{(\xi,u-q,p/2+\frac{\partial}{\partial q}H(u,\xi)/2)|\mbox{ } (u,\xi)\in C_{E,H},\mbox{ } q_0=pr_L\circ S\circ\iota_{E,H}(u,\xi), p=pr_{L^*}\circ S\circ\iota_{E,H}(u,\xi)\}.
\end{equation*}
3) Докажем, что множества $X_{q_0}(P)$ и $X_{q_0}(\widehat{P})$ совпадают. Заметим, что
\begin{equation*}
     \widehat{P}|_{M_{q_0}}= [f(u,\xi)+\phi(q_0)\epsilon(u,\xi)+G(l,t,\chi)-g(v,t)]|_{M_{q_0}}=P(F_{q_0})|_{M_{q_0}}.
\end{equation*}
Таким образом, используя два предыдущих пункта, получаем, что 
\begin{equation*}
    X_{q_0}({P}) = \{(\xi,u-q,p/2+\frac{\partial}{\partial q}F(u,\xi)/2)|\mbox{ } (u,\xi)\in C_{q_0},\mbox{ } p=pr_{L^*}\circ S\circ\iota_{E,F}(u,\xi)\}=  
\end{equation*}
\begin{equation*}
   = \{(\xi,u-q,p/2+\frac{\partial}{\partial q}F_{q_0}(u,\xi)/2)|\mbox{ } (u,\xi)\in \widehat{C}_{q_0},\mbox{ } p=pr_{L^*}\circ S\circ\iota_{E,F_{q_0}}(u,\xi)\} =X_{q_0}({P}(F_{q_0})) =X_{q_0}(\widehat{P}).
\end{equation*}
4) Осталось проверить, что $P$ и $\widehat{P}$ тождественно совпадают на окрестности $X_{q}(P)$. Так как 
\begin{equation*}
    P-\widehat{P}=(1-\phi(q))\epsilon(u,\xi),
\end{equation*}
достаточно для $(q,\xi,l,v,\chi)\in X_{q}(P)$, $q\notin Op(O)$ заметить, что $ (u,\xi)\in C_{q}\subset E-supp\mbox{ }\epsilon.$
\end{proof}
\end{prop}
Теперь мы в качестве следствия получим утверждение, анонсированное выше. 
\begin{prop}
Рассмотрим производящее семейство $(E,F)\in Gf(L)$, разложенное как  
\begin{equation*}
    F=f+\epsilon, \mbox{ где } supp\mbox{ }\epsilon \mbox{ компактен}.
\end{equation*}
Тогда существует расслоенный диффеоморфизм $\psi,$ для которого компактен носитель
\begin{equation*}
    \delta:=\psi^*P-P(f).
\end{equation*}
\begin{proof}
Конструкция проходит в $2$ шага: сначала мы строим $\phi$, для которого $\widehat{P}$ и $P(f)$ совпадают вне компакта, а потом расслоенный диффеоморфизм $\psi$, переводящий $P$ в $\widehat{P}$. \\
1) Построим функцию $\phi$. Так как носитель $\epsilon$ компактен, объединение подмножеств 
\begin{equation*}
    pr_L\circ S\circ\iota_{E,F}(C_{E,F}\cap supp\mbox{ } \epsilon)\cup  pr_L(supp\mbox{ } \epsilon) \subset L
\end{equation*}
содержится в некотором отрезке $I$. Если $\phi\equiv 1$ на некоторой открытой окрестности $I$, то
\begin{equation*}
    O=pr_L\circ S\circ\iota_{E,F}(C_{E,F}\cap supp\mbox{ } \epsilon)\subset I,
\end{equation*}
а значит для $f$, $\epsilon$ и $\phi$ применима предыдущее предложение. Теперь рассмотрим $\phi\in \mathcal{C}^\infty (L)$
\begin{equation*}
    0\le \phi(q)\le 1
\end{equation*}
\begin{equation*}
    \phi|_{Op(I)}\equiv 1, \mbox{ } \phi|_{L-J}\equiv 0
\end{equation*}
для некоторого интервала $J\subset L$, содержащего $Op(I)$. Так как носители $\epsilon$ и $\phi$ компактны, разность $\widehat{P}-P(f)=\phi(q)\epsilon(u,\xi)$ имеет компактный носитель.\\
2) Осталось построить расслоенный диффеоморфизм, переводящий $P$ в $\widehat{P}$. Для семейства 
\begin{equation*}
    \phi_t:=(1-t)\phi+ t,
\end{equation*} 
связывающего $\phi_0\equiv\phi$ и $\phi_1\equiv 1$, выполнены условия ${\phi_t}|_{Op(I)}\equiv 1, \mbox{ } {\phi_t}|_{L-J}\equiv 0,$ так что 
\begin{equation*}
    (M,\widehat{P}_t)\in Gf(L)
\end{equation*}
порождают кривую $\gamma_{M,P}$. При этом, на каждом слое $M_q$ все функции ${\widehat{P}_t}|_{M_q}$ совпадают вне (одного и того же) компакта $supp \mbox{ } \epsilon$. Теперь из гомотопического метода следует, что все $\widehat{P}_t$ эквивалентны друг другу. Проведем это рассуждение до конца.\\
3) Напомним, что все подмногообразия послойных критических точек совпадают 
\begin{equation*}
   C_t=C,\mbox{ где } C_t:=C_{M,\widehat{P}_t}\subset M,
\end{equation*}
причем критические значения $\widehat{P}_t$ равны в каждой точке $C$. Мы построим диффеотопию 
\begin{equation*}
    \psi_t: M\to M, \mbox{ } \psi_t(M_q)=M_q, 
\end{equation*}
\begin{equation*}
    \psi_t^*(\widehat{P}_t)=\widehat{P}_0=\widehat{P}.
\end{equation*}
Семейство $\psi_t$ определяется семейством векторных полей $X_t$, касающихся слоев $M_q$
\begin{equation*}
     \frac{\partial}{\partial t}\psi_t=X_t(\psi_t).
\end{equation*}
Условие $\psi_t^*(\widehat{P}_t)=\widehat{P}_0$ равносильно следующему тождеству на поля $X_t$
\begin{equation*}
    \mathcal{L}_{X_t}\widehat{P}_t=-\frac{\partial}{\partial t}\widehat{P}_t.
\end{equation*}
Мы можем положить $X_t$ равным нулю во всех критических точках $C$, а вне них приравнять
\begin{equation*}
    {X_t}|_{M_q}:=\frac{-\frac{\partial}{\partial t}\widehat{P}_t}{|d{\widehat{P}_t}|_{M_q}|}\nabla ({\widehat{P}_t}|_{M_q}).
\end{equation*}
\end{proof}
\end{prop}

\section{Производящие семейства на Арбореллевских графах}
\label{sec:main}
\subsection*{Точные кривые на перекрестке}
\begin{mydef}
\textbf{Ограниченное производящее семейство} на перекрестке $\boldsymbol{\perp}$ это
\begin{equation*}
   (E,F)\in Gf(O_q),\mbox{ для которого } \gamma_{E,F}\subset \Sigma( \boldsymbol{\perp}).  
\end{equation*}
Множество ограниченных производящих семейств на $\boldsymbol{\perp}$ обозначается за $Gf^c(\boldsymbol{\perp})$. 
\end{mydef}
Кривая $\gamma_{E,F}$, порожденная $(E,F)\in Gf(O_q)$, лежит в $\Sigma(\boldsymbol{\perp})$ тогда и только тогда, когда
\begin{equation*}
    \forall (q,\xi)\in \{F_{\xi}(q,\xi)=0\} \mbox{ выполнено}
\end{equation*}
\begin{equation*}
    \begin{cases}
      F_q(q,\xi)\ge -1,\\
      \begin{array}{|l@{}}
        |q|-1\le 1, \\
         F_q(q,\xi)\le 1. 
    \end{array}
    \end{cases}
\end{equation*}
Сейчас мы наконец определим оператор $\cdot |_{l}:Gf(\boldsymbol{\perp})\to Gf(l)$. Мы сделаем это в два шага: сначала с помощью оператора Чеканова для поворота $W$ определим $\circlearrowright$, а потом локализуем его, пользуясь тем, что все рассматриваемые нами кривые лежат в области $\Sigma( \boldsymbol{\perp})$. 
\begin{proof}[Доказательство предложения~\ref{prop:turn}]
Симплектоморфизм $W$ допускает производящую функцию 
\begin{equation*}
    H=V,\mbox{ } G\in\mathcal{C}^\infty(V),
\end{equation*}
\begin{equation*}
G(u,l):= -\langle u,u\rangle -\langle l,l\rangle. 
\end{equation*}
Действительно, для диффеоморфизма $b:V\to V$, $b(u,l):=(\frac{{l}+{u}}{2},\frac{l-u}{2})$ выполнено тождество 
\begin{equation*}
     \iota_{H,G}(b(u,{l}))=(b(u,{l}),dG(b(u,{l})))=(\frac{u+{l}}{2},\frac{{l}-{u}}{2}, -\widetilde{u}-\widetilde{l}, \widetilde{u}-\widetilde{l} )=  \varpi_\mathbb{V} (u,l, l,-u) = \varpi_\mathbb{V}\circ gr_{W_V}(u,l)\Rightarrow 
\end{equation*}
\begin{equation*}
    \Rightarrow  \iota_{H,G}\circ b= \varpi_\mathbb{V}\circ gr_{W} \Rightarrow \mathcal{L}_S=\iota_{H,G}(V). 
\end{equation*}
Рассмотрим точную кривую $\gamma\subset \mathbb{V}$, допускающую производящее семейство $(E,F)\in Gf(L)$. Применяя конструкцию $Ch$ к симплектоморфизму $W$ и $(E,F)\in Gf(L)$, получаем
\begin{equation*}
    M({E},W)=L_q\times W_{\xi}\times  L_v\times L_t,
\end{equation*}
\begin{equation*}
    P({F},W)(q,\xi,v,t)=F(u,\xi)-g(l,l)-g(t,t)-g(v,t)= F(q+v,\xi)-g(q+v/2,q+v/2)-g(t,t)-g(v,t)= 
\end{equation*}
\begin{equation*}
    = F(q+v,\xi)-g(q,q+v)-g(v/2+t,v/2+t). 
\end{equation*}
Из предложения~\ref{prop:chekanov} следует, что $\gamma_{M,P}= W(\gamma_{E,F})$, так что $\psi_*(E,F)\in Gf(O_p)$ порождает $\gamma$. 
\end{proof}
Выбрасывая из производящего семейства $C_W(E,F)$ лишнюю переменную $v/2+t$, получаем 
\begin{mydef}
\textbf{Оператор поворота} $\circlearrowright: Gf(L)\to Gf(L)$ определяется как
\begin{equation*}
   E^{\circlearrowright}:= L_q\times W_{\xi}\times  L_u, 
\end{equation*}
\begin{equation*}
    F^{\circlearrowright}(q,\xi,u):=F(u,\xi)-g(u,q).
\end{equation*}
\end{mydef}
Сейчас мы изменим определение, используя описанную выше процедуру локализации.\\
Зафиксируем сглаживающую функцию $\phi\in\mathcal{C}^\infty(L)$, заданную следующими условиями
\begin{equation*}
     0\le \phi(q)\le 1,
\end{equation*}
\begin{equation*}
    \phi|_{|q|>2}\equiv 0,\mbox{ } \phi|_{|q|<5/4}\equiv 1. 
\end{equation*}
Рассмотрим производящее семейство $(E,F)\in Gf^c(\boldsymbol{\perp})$. Разобьем функцию $F$ в сумму
\begin{equation*}
     F(q,\xi)=F_c(q,\xi)+F_{nc}(q,\xi):=\phi(q) F(q,\xi) + (1-\phi(q))F(q,\xi). 
\end{equation*}
Определим производящее семейство $(E^\dag, F^\dag)\in Gf(L)$ следующим образом 
\begin{equation*}
    E^\dag:= E^{\circlearrowright}, F^\dag:= \widehat{F^{\circlearrowright}},
\end{equation*}
где в формуле для $\widehat{F^{\circlearrowright}}$ мы используем наше зафиксированное $\phi$. Ее можно записать явно как 
\begin{equation*}
    F^\dag(q,\xi,u):= F_c(u,\xi)+\phi(q)F_{nc}(u,\xi)-g(u,q)= \phi(u) F(u,\xi)+(1-\phi(u))\phi(u)F_{nc}(u,\xi)-g(u,q).
\end{equation*}
\begin{prop}
Для произвольного $(E,F)\in Gf^c(\boldsymbol{\perp})$ кривые $\gamma_{E^\dag, F^\dag}$ и $W(\gamma_{E,F})$ совпадают.
\begin{proof}
Достаточно доказать, что $ \gamma_{E^\dag, F^\dag}= \gamma_{E^\circlearrowright,F^\circlearrowright}.$\\
Воспользуемся предложением~\ref{prop:localization} для $f=F_c$, $\epsilon=F_{nc}$ и заданного $\phi$. По определению,
\begin{equation*}
    O=pr_{L^*}\circ\iota_{E,F_c}(C_{E,F_c}\cap supp\mbox{ } F_{nc})\cup \{r\in L|\mbox{ } \exists (q,\xi)\in C_{E,F_{c,r}}\cap supp\mbox{ } F_{nc}\mbox{ : } pr_{L^*}\circ\iota_{E,F_{c,r}}(q,\xi)=r\}. 
\end{equation*}
Так как $(E,F)\in Gf^c(\boldsymbol{\perp})$ подмножество $\gamma_{E,F_c}$ и все  $\gamma_{E,F_{c,r}}\subset V$ лежат в $\Sigma(\boldsymbol{\perp})$, то
\begin{equation*}
    q\in supp\mbox{ } F_{nc}\Rightarrow |q|>1 \Rightarrow |pr_{L^*}\circ\iota_{E,F_c}|<1 \Rightarrow O\subset [-1,1].
\end{equation*}
Следовательно, $\phi\equiv 1$ на окрестности $O$, и утверждение применимо.
\end{proof}
\end{prop}
\begin{mydef}
\textbf{Оператор ограничения} $\cdot |_{l}$ производящего семейства с $\boldsymbol{\perp}$ на $l$ задается как  
\begin{equation*}
    \cdot |_{l}:Gf^c(\boldsymbol{\perp})\to Gf(O_{p>0})
\end{equation*}
\begin{equation*}
    (E,F)|_{l}:=(\psi_*(E^\dag, F^\dag))|_{O_{p>0}}
\end{equation*}
\end{mydef}
В дальнейшем, нам понадобится обратить $\dag$. На первый взгляд, это невозможно, так как оператор Чеканова увеличивает ранг расслоения. Но мы изучаем семейства с точностью до стабильной эквивалентности, так что нам достаточно стабильной обратимости. Рассмотрим оператор $-\dag$, построенный по описанной выше схеме но для симплектоморфизма $W^{-1}$.\\
Заметим, что производящая функция $-G$ порождает симплектоморфизм $W^{-1}$, так что 
\begin{equation*}
    F^{\circlearrowleft}(q,\xi,u):=F(u,\xi)+g(u,q).
\end{equation*}
\begin{prop}
Для любого семейства $(E,F)\in Gf^c(\boldsymbol{\perp})$ существует эквивалентность
\begin{equation*}
    ((E,F)^\dag)^{-\dag}\cong(E,F)\oplus (L_v\oplus L_l,g(v,v)+g(l,l)).
\end{equation*}
\begin{proof}
Достаточно проверить это тождество для операторов $\circlearrowright$ и $\circlearrowleft$. В этом случае
\begin{equation*}
    (F^\circlearrowright)^{\circlearrowleft}(q,\xi,u,v)=F(u,\xi)-vu+vq=F(u,\xi)+v(q-u).
\end{equation*}
Заменяя $u$ на $u+q$ получаем обычную формулу для $P(F,Id_V)$.
\end{proof}
\end{prop}
\subsection*{Производящие семейства на Арбореллевских графах}
Зафиксируем Арбореллевский граф $\mathcal{T}$, Лиувиллеву поверхность $\mathbb{S}(\mathcal{T})$ и ее подобласть $\Sigma(\mathcal{T})$.\\
Напомним, что $\Sigma(\mathcal{T})$ покрывается стандартными картами $W_{v}$, пронумерованными $\mathcal{V}(\Gamma)$.\\
Рассмотрим набор $(E_v,F_v)\in Gf(I_v)$ производящих семейств, такой, что для каждого ребра $e$, соединяющего вершины $v$ и $w$, зафиксирована стабильная эквивалентность ограничений 
\begin{equation*}
    (E_v,F_v)|_{J_e}\cong^{st} (E_w,F_w)|_{J_e}.
\end{equation*}
Заметим, что кривые $\gamma_{E_v,F_v}$ совпадают на пересечениях карт и склеиваются в кривую $\gamma_{\mathcal{E},\mathcal{F}}$. В таком случае говорят, что $(\mathcal{E},\mathcal{F})$ порождает кривую $\gamma_{\mathcal{E},\mathcal{F}}$. 
\begin{mydef}
\textbf{Ограниченное производящее семейство} на Арбореллевском графе $\mathcal{T}$ это 
\begin{equation*}
    (\mathcal{E},\mathcal{F})\in \sqcup_{v\in\mathcal{V}(\Gamma)} Gf(I_v)
\end{equation*}
которое порождает кривую $\gamma_{\mathcal{E},\mathcal{F}}$, лежащую в  $\Sigma(\mathcal{T})$. Они образуют множество $Gf^c(\mathcal{T})$.
\end{mydef}
Рассмотрим подгруппу $Ham(\Sigma(\mathcal{T}))\subset Ham^c(\mathbb{S}(\mathcal{T}))$ гамильтоновых симплектоморфизмов, носитель которых лежит в $\Sigma(\mathcal{T})$ и не пересекает $\mathbb{D}_v-R_e$, если $e$ это единственное ребро $v$.\\
Для доказательства теоремы~\ref{theorem:main} нам понадобится следующая версия Леммы о фрагментации.
\begin{lem}[\cite{Ban}]
Любой элемент $\psi\in Ham(\Sigma(\mathcal{T}))$ можно разложить в композицию 
\begin{equation*}
   \psi=\psi_n\circ\cdots\circ\phi_1, \mbox{ }\psi_i\in Ham(\Sigma(\mathcal{T}))
\end{equation*}
так, что $\psi_i$ $C^1$-близко к $id$ и имеет компактный носитель, лежащий в карте $W_{v_i}$.
\end{lem}
\begin{proof}[Доказательство теоремы~\ref{theorem:main}]
Зафиксируем кривую $\gamma$, допускающую производящее семейство. Мы будем доказывать утверждение теоремы по индукции по длине кратчайшего разложения 
\begin{equation*}
     \psi=\psi_n\circ\cdots\circ\psi_1
\end{equation*}
Заметим, что шаг и база индукции совпадают с частным случаем утверждения для 
\begin{equation*}
   \psi\in Ham(W_v), \mbox{ } C^1 \mbox{-близкого к } Id.
\end{equation*}  
Рассмотрим образ кривой $\delta:=\psi(\gamma)$. Мы хотим построить семейства $(H_w,G_w)$, порождающие
\begin{equation*}
    \delta_w\subset W_w
\end{equation*} 
так, чтобы они были стабильно эквивалентны друг другу на пересечениях. Рассмотрим форму 
\begin{equation*}
    q:V_{x,y}\to\mathbb{R}, \mbox{ } q(x,y):=g(x,x)+g(y,y). 
\end{equation*}
Следующая лемма сразу следует из свойства гамильтонова подъема для семейств над $L$ в той форме, в которой оно доказано в статье \cite{EG04}.
\begin{lem}
Рассмотрим Гамильтонову изотопию $\psi\in Ham(\Sigma(\boldsymbol{\perp}))\mbox{ } C^1 \mbox{-близкую к } Id.$ Для $(E,F)\in Gf(\boldsymbol{\perp})$ существует функция $h$ на $H:=E\times V_{x,y}$ с компактным носителем, что 
\begin{equation*}
    \phi(\gamma_{E,F})=\gamma_{H,G}, \mbox{ где }  G(q,\xi,v,t):=F(q,\xi)+q(x,y)+h(q,\xi,v,t). 
\end{equation*}
\end{lem}
Используя Лемму, заменим семейство $(E_v,F_v)$ на новое семейство, порождающее $\delta_v$
\begin{equation*}
    (H_v:=E_v\times V,G_v:=F_v+q+h).
\end{equation*}
Пусть $d(v)=3$. Напомним, что для из компактности носителя функции $h\in\mathcal{C}^\infty(E_v)$ следует существование расслоенного диффеоморфизма $\psi:M\to M,$ такого, что носитель разности $\delta:=\psi^*(F+h)^\dag-F^\dag$ компактен. Подкрутим тривиализацию ${H_v}|_{I_l(v)}$ на диффеоморфизм $\psi$.\\
Теперь построим $(H_w,G_w)$ для остальных вершин $w$. Если $w$ и $v$ не соединены ребром, $U_w$ не пересекается с носителем $\psi$ и кривые $\delta_{w}$ и $\gamma_{w}$ совпадают. В этом случае, положим 
\begin{equation*}
    (H_w:=E_w\times V, \mbox{ } G_w:=F_w+q).
\end{equation*}
Осталось разобраться с вершинами, соседними с $v$.  Для каждой такой вершины $w$ пересечение $W_w$ с носителем $\psi$ содержится в ленте, отвечающей их общему ребру $e$. На $J_e$ уже задано производящее семейство для кривой $\delta$. Так как $\dag$ переводит стабильно эквивалентные семейства в стабильно эквивалентные,  это семейство отличается от $F_w|_{J_e}+h$ на функцию $\epsilon_w$, носитель которой лежит вне $\mathbb{D}_w$. Мы воспользуемся следующей  Леммой
\begin{lem}
Рассмотрим перекресток $\boldsymbol{\perp}$, одно из его ребер $e$ и пару производящих семейств
\begin{equation*}
    (E,F)\in Gf^c(\boldsymbol{\perp}),\mbox{ } (H,G)\in Gf(I_e).
\end{equation*} 
Пусть стабилизация ограничения $(E,F)|_{I_e}\oplus(V,q)$ совпадает с семейством $(H,G)$ на 
\begin{equation*}
     J_e\cap K, \mbox{ } v\in K
\end{equation*}
пересечении $J_e$ и некоторого подинтервала $K$ ребра $e$, содержащего $v$. Тогда существует
\begin{equation*}
    (\widetilde{E},\widetilde{F})\in Gf^c(\boldsymbol{\perp}),
\end{equation*}
ограничение которого на $e$ удовлетворяет $(\widetilde{E},\widetilde{F})|_{J_e}\cong^{st} (H,G)|_{J_e},$ а на другие ребра
\begin{equation*}
   (\widetilde{E},\widetilde{F})|_{J_{e'}}\cong (E,F)|_{J_{e'}}\oplus (V,q).
\end{equation*}
\end{lem}
\begin{proof}[Доказательство Леммы]
Либо $e$ является ножкой вершины $v$, либо не является. \\
1) Пусть $e$ не ножка вершины $v$. Тогда на $I_v$ заданы два производящих семейства
\begin{equation*}
    (E,F)\oplus(V,h)\mbox{ и } (H,G),
\end{equation*}
совпадающие на окрестности $K=Op(v)$. Склеим их в одно семейство $(\widetilde{E},\widetilde{F})$, совпадающее с $(H,G)$ на всем $J_e$ и с $(E,F)\oplus(V,h)$ на втором из ребер, образующих шляпку $v$. Осталось проверить утверждение об ограничении на $l(v)$. Из определения сразу следует, что 
\begin{equation*}
    \widetilde{F}^\dag|_{J_{l(v)}}= (\widetilde{F}_c)^\dag|_{J_{l(v)}}= (F+h)^\dag|_{J_{l(v)}}=F^\dag|_{J_{l(v)}}+h.
\end{equation*}
2) Пусть теперь  $d(v)=3$, $e=l(v)$. Рассмотрим семейство $(E,F)^\dag\oplus(V,h)$ и склеим его с $(H,G)$ также, как мы делали выше. Применяя к склеенному семейству $(\widetilde{H},\widetilde{G})$ оператор $-\dag$ получаем
\begin{equation*}
    (\widetilde{E},\widetilde{F}):=(\widetilde{H},\widetilde{G})^{-\dag}. 
\end{equation*}
Осталось проверить условия на ограничения. Для ограничений на ребра, входящие в шляпку 
\begin{equation*}
    \widetilde{F}|_{J_{e}}= \widetilde{G}^{-\dag}|_{J_{e}}=(\widetilde{G}_c)^{-\dag}|_{J_{e}} = (({F})^\dag)^{-\dag}|_{J_{e}} \cong^{st}F|_{J_{e}}\oplus (V,h).
\end{equation*}
Для ограничения на шляпку, получаем следующее тождество $  \widetilde{F}|_{J_{l(v)}}= (\widetilde{G}^{-\dag})^\dag|_{J_{l(v)}} \cong^{st}G|_{J_{l(v)}}.$
\end{proof}
Применяя эту Лемму к каждой из вершин $w$, соседних с $v$, определяем семейство $(H_w,G_w)$. Осталось заметить, что из нашей конструкции следует два утверждения:
\begin{itemize}
    \item ограничения семейств $(H_w,G_w)$ стабильно изоморфны друг-другу на пересечениях карт;
    \item и кроме того каждое семейство $(H_w,G_w)$ порождает $\delta_w$ на своей карте $U_w$.
\end{itemize} 
Первое утверждение достаточно проверять для пар вершин, хотя бы одна из которых соединена ребром с $v$. Действительно, если обе вершины не соседние с $w$, то мы просто стабилизируем соответствующие им производящие семейства, и они остаются стабильно эквивалентны. Если $w$ соседняя с $v$, то для нее это следствие Леммы. \\
С учетом предыдущего пункта второе утверждение достаточно проверять для вершин, не соседних с $v$. Для них это следствие конструкции, которое мы уже обсудили выше.\\
Таким образом, теорема о поднятии Гамильтоновой изотопии доказана. 
\end{proof}

И. Яковлев
\noindent\textsc{Center for Advanced Studies, Skoltech, Moscow, Russia,\\
International laboratory of mirror symmetry and automorphic forms\\  National Research University Higher School of Economics, Moscow, Russia}

\emph{E-mail}:\,\,\textbf{iayakovlev\_1@edu.hse.ru}
\end{document}